\documentclass[12pt,a4paper]{amsart}
\usepackage{amsmath,amssymb,amsthm,latexsym}
\usepackage[utf8]{inputenc}
\usepackage{tikz-cd}
\usepackage[colorlinks=true]{hyperref}
\usepackage[a4paper,portrait]{geometry}
\usepackage{todonotes}
\usepackage{graphicx}
\usepackage{color}
\usepackage{subfigure}
\usepackage{makecell}
\usepackage[section]{placeins}

\numberwithin{equation}{section}

\newcommand{\N}{\mathbb{N}}
\newcommand{\Z}{\mathbb{Z}}
\newcommand{\Q}{\mathbb{Q}}

\newcommand{\R}{\mathbb{R}}
\newcommand{\CC}{\mathbb{C}}

\DeclareMathOperator{\Gal}{Gal}

\DeclareMathOperator{\im}{Im}

\DeclareMathOperator{\re}{Re}

\DeclareMathOperator{\SL}{SL}

\newtheorem{thm}{Theorem}[section]
\newtheorem*{thm*}{Theorem}

\newtheorem*{cor*}{Corollary}
\newtheorem{conjecture}[thm]{Conjecture}
\newtheorem{algo}[thm]{Algorithm}
\theoremstyle{definition}

\theoremstyle{remark}
\newtheorem{remark}[thm]{\bf Remark}

\begin{document}

\author[Zhengyu Tao]{Zhengyu Tao}


\address{Department of Mathematics, Nanjing University, Nanjing 210093, People's Republic of China}

\email{taozhy@smail.nju.edu.cn}

\author[Xuejun Guo]{Xuejun Guo}

\address{Department of Mathematics, Nanjing University, Nanjing 210093, People's Republic of China}

\email{guoxj@nju.edu.cn}

\author[Tao Wei]{Tao Wei}

\address{Department of Mathematics, Nanjing University, Nanjing 210093, People's Republic of China}

\email{weitao@smail.nju.edu.cn}

\title[Mahler measures and $L$-values]{Mahler measures and $L$-values of elliptic curves over real quadratic fields}




\thanks{The authors are supported by NSFC 11971226 and NSFC 12231009}

\begin{abstract}
	A famous formula of Rodriguez Villegas shows that the Mahler measures $m(k)$ of $P_k(x,y)=x+1/x+y+1/y+k$ can be written as a Kronecker-Eisenstein series. We prove that the degree of $k$ in Villegas' formula can be bounded by the class numbers of CM points. This fact allows us to systematically derive $28$ new identities linking $m(k)$ to $L$-values of cusp forms. Guided by Beilinson's conjecture, we also prove $5$ formulas that express  $L$-values of CM elliptic curves over real quadratic fields to some $2\times 2$ determinants of $m(k)$. This extends a recent work of Guo (the second author of this paper), Ji, Liu, and Qin, in which they dealt with the cases when $k=4\pm 4\sqrt{2}$.
\end{abstract}

\maketitle

\section{Introduction}
\label{intro}


The \textit{(logarithmic) Mahler measure} of a Laurent polynomial $P\in\CC[x_1^{\pm 1},\cdots,x_n^{\pm 1}]$ is defined by

\[m(P)=\int_0^1\cdots\int_0^1\log|P(e^{2\pi i\theta_1},\cdots,e^{2\pi i\theta_n})|d\theta_1\cdots d\theta_n.\]
There are many connections between the Mahler measures of certain polynomial families associated to elliptic
curves and special values of their $L$-functions. For example, by using Beilinson's regulator map,  Deninger \cite{Den97} conjectured that
\begin{equation}\label{Deningerconjecture}
	m\left(x+\frac{1}{x}+y+\frac{1}{y}+1\right)=L'(E_1,0)=\frac{15}{4\pi^2}L(E_1,2),
\end{equation}
where $E_1$ is the elliptic curve of conductor $15$ given by the projective closure of
\[x+\frac{1}{x}+y+\frac{1}{y}+1=0.\]

Equation \eqref{Deningerconjecture} is the first conjectural formula that connects Mahler measure and $L$-function of elliptic curve. Note that $E_1$ does not have CM. Hence, it is essentially difficult to prove  \eqref{Deningerconjecture}. Almost two decades after it was proposed, this conjecture was finally proved by Rogers and Zudilin in \cite{RZ14}.

In this paper, we study the Mahler measures of  the polynomial family
\[P_k(x,y)=x+\frac{1}{x}+y+\frac{1}{y}+k, \]
where  $k\in\CC\setminus\{0,\pm 4\}$. One can easily check that the rational transformation
\begin{equation}\label{rationaltransformation}
	x=\frac{kX-2Y}{2X(X-1)},\quad y=\frac{kX+2Y}{2X(X-1)}
\end{equation}
converts the zero locus of $P_k(x,y)=0$ to an elliptic curve $E_k$ with Weierstrass equation
\begin{equation}\label{Weierstrassform}
	Y^2=X^3+\left(\frac{k^2}{4}-2\right)X^2+X.
\end{equation}
Note that $E_k$ is defined over $\Q$ if $k^2\in\Q$. The inverse of transformation \eqref{rationaltransformation} is
\begin{equation}\label{inverserationaltransformation}
	X=-\frac{1}{xy},\quad Y=\frac{(y-x)(1+xy)}{2x^2y^2}.
\end{equation}

For simplicity, we write $m(k)=m(P_k)$ in the rest of this paper. Motivated by Deninger's conjecture, Boyd \cite{Boyd98} did a lot of numerical calculations and conjectured that for many $k\in \Z\setminus\{0,\pm 4\}$, $m(k)$ is a nonzero rational multiple of $L'(E_k,0)$, i.e.
\begin{equation}\label{mahlerandelliptic}
	m(k)=r_kL'(E_k,0)
\end{equation}
for some $r_k\in\Q^*$. Boyd listed his conjectural identities in a table of \cite{Boyd98}, he numerically verified those identities with at least 25 decimal places of precision. The famous modularity theorem tells us that if $E/\Q$ is an elliptic curve of conductor $N$, then there exists a newform $f\in\mathcal S_2(\Gamma_0(N))$ such that
\[L(E,s)=L(f,s).\]
According to the functional equation satisfied by $L(E,s)$, we have
\[L'(E,0)=\pm\frac{N}{4\pi^2}L(E,2).\]
Thus we can rewrite \eqref{mahlerandelliptic} as
\begin{equation}\label{mahlermeasureasL2}
	m(k)=r_kL'(f_N,0)=r_k\frac{N}{4\pi^2}L(f_N,2),
\end{equation}
where $f_N\in\mathcal S_2(\Gamma_0(N))$ is the newform associated to $E_k$.

Recall that the classical Dedekind eta-function is the infinite product
\[\eta(\tau)=e^{\frac{2\pi i\tau}{24}}\prod\limits_{n=1}^{\infty}(1-q^{n}),\]
where $\tau\in\mathcal H=\{z\in\CC|\im(z)>0\}$ and $q=e^{2\pi i\tau}$. Let $\chi_{-d}=\left(\frac{-d}{n}\right)$ be the real odd Dirichlet character of conductor $d$ given by the Kronecker-Jacobi symbol $\left(\frac{\cdot}{\cdot}\right)$. The modular lambda-function $\lambda(\tau)$ is defined by
\begin{equation}\label{lambda}
	\lambda(\tau)=16\frac{\eta(\tau/2)^8\eta(2\tau)^{16}}{\eta(\tau)^{24}}.
\end{equation}
In \cite{Vil97}, Villegas proved the following result that express $m(k)$ as a Kronecker-Eisenstein series:
\begin{thm}\label{mahlermeasureaslambda}
	If $\tau$ lies in the region $\mathcal F'\subset\mathcal H$ formed by the geodesic triangle of vertices $ i\infty,0,1/2$ and its reflection along the imaginary axis, then we have
	\begin{equation}\label{mahlermeasureandlambda}
		m(k)=\re \left(\frac{16\im(\tau)}{\pi^2}\underset{m,n\in\Z}{{\sum}'}\frac{\chi_{-4}(n)}{(4m\tau+n)^2(4m\bar\tau+n)}\right),
	\end{equation}
	where $k=\frac{4}{\sqrt{\lambda(2\tau)}}$ and $\underset{m,n\in\Z}{{\sum}'}$ means $(0,0)$ is excluded from the summation.
\end{thm}

Then, by taking $\tau$ to be some appropriate CM points (quadratic algebraic numbers in $\mathcal H$) in \eqref{mahlermeasureandlambda}, he proved that
\begin{equation}\label{Vliegasexamples}
	m(4\sqrt{2})=L'(E_{4\sqrt{2}},0),\ \ \ m(2\sqrt{2})=L'(E_{2\sqrt{2}},0),\ \ \ m(4 i)=2L'(E_{4 i},0).
\end{equation}

Although much research has been done, only a few $k$ with $E_k$ defined over $\Q$ are rigorously proved to satisfy \eqref{mahlerandelliptic}. Samart made a comprehensive table of these $k$ in \cite{Sam21}, they are
\[k=1,3 i,5,16, i,3,2,3\sqrt{2},8, i\sqrt{2},4 i,2\sqrt{2},2 i,12,\sqrt{2},4\sqrt{2}.\]
He also listed the conductors of each $E_k$, the rational numbers $r_k$, and their references. See \cite{Bru16, Lal10,LR07,LSZ16,RZ14,Sam15,Zud14} for more about this subject.

One may wonder what will happen if $E_k$ is defined over a number field $K$ with $r=[K:\Q]>1$. In fact, by Theorem \ref{mahlermeasureaslambda}, if there is a CM point $\tau\in\mathcal{F}'$ such that $k=\frac{4}{\sqrt{\lambda(2\tau)}}$, then we can still write $m(k)$ as a lattice sum, and thus it can also be written as $L$-value of a modular form. However, in order to relate $m(k)$ to the $L$-value of $E_k$, we need to turn to a $r\times r$ determinant with $m(k)$ as its entries according to Beilinson's conjecture.

For example, in the real quadratic case $k=4\pm 4\sqrt{2}$, by taking $\tau=\frac{i\sqrt{2}}{2}$ and $i\sqrt{2}$ into \eqref{mahlermeasureandlambda}, Guo, Ji, Liu, and Qin proved in \cite{Guo} that
\begin{equation}\label{MMascuspform}
	m(4+4\sqrt{2})=\frac{16\sqrt{2}}{\pi^2}L(f,2),\quad m(4-4\sqrt{2})=\frac{16\sqrt{2}}{\pi^2}L(g,2),
\end{equation}
where $f(\tau)=\frac{\eta(8\tau)\eta(16\tau)^5}{\eta(32\tau)^2}$ and $g=2\frac{\eta(8\tau)^3\eta(32\tau)^2}{\eta(16\tau)}$ are weight $2$ cusp forms for $\Gamma_1(64)$. Then, by constructing the associated Beilinson regulator, they proved that
\begin{equation}\label{theGJLQcase}
	\left|\det{\begin{pmatrix}m(4+4\sqrt{2})& m(4-4\sqrt{2})\\
			m(4-4\sqrt{2})&-m(4+4\sqrt{2})\end{pmatrix}}\right|=\frac{512}{\pi^4}L(E_{4\pm4\sqrt{2}},2).
\end{equation}

The present paper is intended to extend the work of \cite{Guo}. Since $\lambda(2\tau)$ is a modular function for $\Gamma_0(4)$ with rational Fourier coefficients. It is known that $\lambda(2\tau)$ are algebraic numbers when $\tau$ takes CM points. Moreover, the degree of $\lambda(2\tau)$ have something to do with the class number of $D_\tau$, the discriminant of $\tau$. We will explore the connection between them and this will help us find some CM points that make $\lambda(2\tau)$ look good. Then we will use these points to prove identities similar to \eqref{MMascuspform}. We summarize the results in the following theorem.
\begin{thm}\label{MMasLofcuspforms}
	We have identities of the form
	\[m(k)=\frac{c_k}{\pi^2}L(f_k,2)\]
	for the following $28$ $k$:
	\[12\pm 8\sqrt{2},\quad 8i\sqrt{3\sqrt{2}+4},\quad 8\sqrt{3\sqrt{2}-4},\quad 2i\sqrt{2\sqrt{2}-2},\quad 2\sqrt{2\sqrt{2}+2}\]
	\[4i\sqrt{2\sqrt{2}+2},\quad 4\sqrt{2\sqrt{2}-2},\quad (8\pm4\sqrt{2})i,\quad \sqrt{2}\pm\sqrt{6},\quad 4\sqrt{2}\pm4\sqrt{6},\quad 2\sqrt{3}\pm 2i\]
	\[(32\sqrt{2}\pm 12\sqrt{7})i,\quad \frac{3\sqrt{2}}{2}\pm\frac{\sqrt{14}}{2},\quad 24\sqrt{2}\pm8\sqrt{14},\quad 6\pm 2i\sqrt{7},\quad \frac{3}{2}\pm i\frac{\sqrt{7}}{2},\quad \frac{3\sqrt{7}}{2}\pm\frac{i}{2}.\]
	Where $f_k$ are normalized weight $2$ cusp forms for $\Gamma_0(N_k)$ of the form
	\[r\underset{m,n\in\Z}{{\sum}}\chi_{-4}(n)(lm+sn)q^{am^2+bmn+cn^2},\]
	$c_k$ are quadratic numbers such that $c_k^2\in\Q$. To save space, we list the explicit expressions of $f_k$, the levels $N_k$ and $c_k$ in Table \ref{mahlermeasureandmodularforms}.
\end{thm}

When $E_k$ is defined over a number field $K$, the approach of Guo et al. used to construct the Beilinson regulator requires that the rational transformation \eqref{rationaltransformation} is also defined over $K$. One can check that it is the case for $k=12\pm 8\sqrt{2}$. So we can follow their approach and prove the following formula.

\begin{thm}\label{mainresults1}
	We have
	\begin{equation}\label{thecasek=12pm8sqrt2}
		\left|\det{\begin{pmatrix}m(12+8\sqrt{2})& m(12-8\sqrt{2})\\
				m(12-8\sqrt{2})&m(12+8\sqrt{2})\end{pmatrix}}\right|=\frac{1024}{\pi^4}L(E_{12\pm8\sqrt{2}},2).
	\end{equation}
\end{thm}
For the cases when $k=\sqrt{2}\pm\sqrt{6},\ 4\sqrt{2}\pm4\sqrt{6},\ \frac{3\sqrt{2}}{2}\pm\frac{\sqrt{14}}{2},\ 24\sqrt{2}\pm8\sqrt{14}.$ Although $E_k$ will still be defined over real quadratic fields. The rational transformations \eqref{rationaltransformation} are no longer defined over the same field. However, by changing the Weierstrass model of $P_k(x,y)=0$ to
\begin{equation*}
	E_k':Y^2=X^3+\left(\frac{k^2}{2}-4\right)X^2+4X,
\end{equation*}
we can still prove $2\times 2$ formulas for these $k$:
\begin{thm}\label{mainresults2}
	Let $E_k'$ be the curve defined above, then we have
	\begin{equation}\label{thecasek=sqrt2pmsqrt6}
		\left|\det{\begin{pmatrix}m(\sqrt{2}+\sqrt{6})& m(\sqrt{2}-\sqrt{6})\\
				3m(\sqrt{2}-\sqrt{6})&-m(\sqrt{2}+\sqrt{6})\end{pmatrix}}\right|=\frac{144}{\pi^4}L(E'_{\sqrt{2}\pm\sqrt{6}},2),
	\end{equation}
	\begin{equation}\label{thecasek=4sqrt2pm4sqrt6}
		\left|\det{\begin{pmatrix}
				m(4\sqrt{2}+4\sqrt{6}) & m(4\sqrt{2}-4\sqrt{6}) \\
				m(4\sqrt{2}-4\sqrt{6}) & -3m(4\sqrt{2}+4\sqrt{6})\end{pmatrix}}\right|=\frac{2304}{\pi^4}L(E'_{4\sqrt{2}\pm4\sqrt{6}},2),
	\end{equation}
	\begin{equation}\label{thecasek=3sqrt2/2pmsqrt14/2}
		\left|\det{\begin{pmatrix}
				m\big(\frac{3\sqrt{2}}{2}+\frac{\sqrt{14}}{2}\big) & m\big(\frac{3\sqrt{2}}{2}-\frac{\sqrt{14}}{2}\big) \\
				7m\big(\frac{3\sqrt{2}}{2}-\frac{\sqrt{14}}{2}\big) & -m\big(\frac{3\sqrt{2}}{2}+\frac{\sqrt{14}}{2}\big)\end{pmatrix}}\right|=\frac{196}{\pi^4}L(E'_{\frac{3\sqrt{2}}{2}\pm\frac{\sqrt{14}}{2}},2),
	\end{equation}
	\begin{equation}\label{thecasek=24sqrt2pm8sqrt14}
		\left|\det{\begin{pmatrix}
				m(24\sqrt{2}+8\sqrt{14}) & m(24\sqrt{2}-8\sqrt{14}) \\
				m(24\sqrt{2}-8\sqrt{14}) & -7m(24\sqrt{2}+8\sqrt{14})\end{pmatrix}}\right|=\frac{12544}{\pi^4}L(E'_{24\sqrt{2}\pm8\sqrt{14}},2).
	\end{equation}
\end{thm}

This paper is organized as follows.

In Section \ref{CNADOL}, we study the modular function $\lambda(2\tau)$ for $\Gamma_0(4)$. It is proved that when $\tau_0$ is a CM point with discriminant $D_{\tau_0}$, then $\lambda(2\tau_0)$ will be an algebraic number with degree no more than $h(D_{\tau_0})h(D_{4\tau_0})$, where $h(D)$ is the class number of primitive binary quadratic forms with discriminant $D$. 

In Section \ref{searchcmpoints}, we will use all negative discriminants $D$ with $h(D)\leqslant 2$ and an algorithm to obtain a list (see Table \ref{cmpoints}) of CM points in $\mathcal F'$ with $h(D_{\tau_0})h(D_{4\tau_0})\leqslant 4$ and distinct $\lambda(2\tau_0)$ values. Our idea is partly inspired by the work \cite{Piseries} of  Huber, Schultz, and Ye, in which they searched CM points to derive many Ramanujan-Sato series for $1/\pi$ of level $20$.

In Section \ref{tableofmahlermeasure}, we use the CM points in Table \ref{cmpoints} to prove Theorem \ref{MMasLofcuspforms}. Instead of providing the details for all cases, we only prove the case when $k=12\pm 8\sqrt{2}$ as an illustration and list other results directly in Table \ref{mahlermeasureandmodularforms}. Once we know the CM points, the other cases can be proved in the same way. In addition to identities in Theorem \ref{MMasLofcuspforms}, we also find that Table \ref{mahlermeasureandmodularforms} reproduces all proven cases of $m(k)$ when $E_k$ have CM.

In Section \ref{Beilinson}, we briefly introduce Beilinson's conjecture for curves over number fields. Our references are \cite{DJZ06,LJ15}.

In Section \ref{example12pm8sqrt(2)}, we prove Theorem \ref{mainresults1}. Vélu's formula will be used to construct the isogeny between $E_{12+8\sqrt{2}}$ and $E_{12-8\sqrt{2}}$. This will help us to determine the shape of Beilinson regulator which is a $2\times2$ determinant of $m(12\pm 8\sqrt{2})$. When it comes to the $L$-functions of $E_{12\pm 8\sqrt{2}}$, the LMFDB database \cite{database} helps a lot.

In Section \ref{other4cases}, we prove Theorem \ref{mainresults2}. When dealing with these cases, we need to turn to the Weierstrass model $E'_k$.

The computations were mainly performed in \textsf{Mathematica}. We also used \textsf{SageMath} and \textsf{PARI/GP} for elliptic curve-related calculations.

\section{The degree of $\lambda(2\tau)$}\label{CNADOL}

Since according to \eqref{lambda}, $\lambda(2\tau)$ has the eta-quotient expression
\[\lambda(2\tau)=16\frac{\eta(\tau)^8\eta(4\tau)^{16}}{\eta(2\tau)^{16}}=16q-128q^2+704q^3-3072q^4+\cdots,\]
we can deduce (for example, from \cite[Proposition 5.9.2]{Mfaca}) that $\lambda(2\tau)$ is a modular function for $\Gamma_0(4)$.

By the theory of
complex multiplication, we know that every modular function for $\Gamma_0(m)$ is a rational function of $j(\tau)$ and $j(m\tau)$, where
\[j(\tau)=\frac{1}{q}+744+196884q+21493760q^2+\cdots\]
is the modular invariant. To make it more clear, let
\[C(m)=[\SL(2,\Z):\Gamma_0(m)]=m\prod_{p|m}\left(1+\frac{1}{p}\right),\]
and $\gamma_i,\ i=1,\cdots,C(m)$ be the right coset representatives of $\Gamma_0(m)$ in $\SL(2,\Z)$.
For each $m$, there exists a polynomial (usually called the \textit{modular equation} for $\Gamma_0(m)$)
\[\Phi_m(X,Y)\in\Z[X,Y]\]
of degree $C(m)$ that satisfies
\begin{equation}\label{PHIm}
	\Phi_m(X,j(\tau))=\prod_{i=1}^{C(m)}(X-j(m\gamma_i\tau)).
\end{equation}
\begin{thm}[{\cite[Proposition 12.7]{Cox}}]\label{lambda2}
	Let $f(\tau)$ be a modular function for $\Gamma_0(m)$ whose $q$-expansion has rational coefficients. Then
	\begin{enumerate}
		\item $f(\tau)\in\Q(j(\tau),j(m\tau))$.
		\item Assume in addition that $f(\tau)$ is holomorphic on $\mathcal{H}$, and let $\tau_0\in\mathcal{H}$. If
		\[\frac{\partial\Phi_m}{\partial X}(j(m\tau_0),j(\tau_0))\neq 0,\]
		then $f(\tau_0)\in \Q(j(\tau_0),j(m\tau_0)).$
	\end{enumerate}
\end{thm}

Now, back to the case $m=4$ that we are interested in, we have $[\SL(2,\Z):\Gamma_0(4)]=6$. Let $S=\begin{pmatrix}
	0 & -1\\
	1 & 0
\end{pmatrix}, T=\begin{pmatrix}
	1 & 1\\
	0 & 1
\end{pmatrix}$ be the generators of $\SL(2,\Z)$. One can verify that
\begin{equation}\label{rightcosets}
	\gamma_1=I_2,\ \gamma_2=S,\ \gamma_3=ST,\ \gamma_4=ST^{-1},\ \gamma_5=ST^{-2},\ \gamma_6=ST^{-2}S
\end{equation}
is a set of right coset representatives of $\Gamma_0(4)$ in $\SL(2,\Z)$.

For a negative integer $D$ with $D\equiv 0 \text{ or } 1 \mod 4$, let $h(D)$ be the class number of primitive binary quadratic forms of
discriminant $D$. By a \textit{CM point}, we mean a point $\tau_0\in \mathcal{H}$ that satisfies $a\tau_0^2+b\tau_0+c=0$, where $a,b,c$ are integers with $a>0$ and $\gcd(a,b,c)=1.$ Denote $ D_{\tau_0}:=b^2-4ac$ be the \textit{discriminant} of $\tau_0$. It is convenient to write $\tau_0$ as $(a,b,c)$.

\begin{thm}\label{boundbyclassnumber}
	Let $\tau_0\in\mathcal{H}$ be a CM point. If
	\[\prod_{i\neq 1}\big(j(4\tau_0)-j(4\gamma_i\tau_0)\big)\neq 0,\]
	where $\gamma_i$ are the coset representatives in \eqref{rightcosets}, then $\lambda(2\tau_0)$ is an algebraic number with degree no more than $h(D_{\tau_0})h(D_{4\tau_0})$.
\end{thm}
\begin{proof}
	It is clear that $\displaystyle \lambda(2\tau)=16\frac{\eta(\tau)^8\eta(4\tau)^{16}}{\eta(2\tau)^{16}}$ is holomorphic on $\mathcal{H}$ since
	\[\eta(\tau)=e^{\frac{2\pi i\tau}{24}}\prod\limits_{n=1}^{\infty}(1-q^{n})\]
	has no zeros on $\mathcal{H}$. Take the derivative of both sides of \eqref{PHIm} with respect to $X$ and evaluate at $\tau=\tau_0$, we have
	\[\frac{\partial\Phi_4}{\partial X}(j(4\tau_0),j(\tau_0))=\prod_{i\neq 1}\big(j(4\tau_0)-j(4\gamma_i\tau_0)\big).\]
	Thus, when the above equation is not zero, we know form Theorem \ref{lambda2} that \[\lambda(2\tau_0)\in\Q(j(\tau_0),j(4\tau_0)).\]
	Since $j(\tau_0)$ is an algebraic integer of degree $h(D_{\tau_0})$ by the theory of complex multiplication, we have
	\begin{align*}
		[\Q(\lambda(2\tau_0)):\Q]&\leqslant [\Q(j(\tau_0),j(4\tau_0)):\Q]\\
		&= [\Q(j(\tau_0),j(4\tau_0)):\Q(j(4\tau_0))][\Q(j(4\tau_0)):\Q]\\ 
		&\leqslant[\Q(j(\tau_0)):\Q][\Q(j(4\tau_0)):\Q]\\
		&=h(D_{\tau_0})h(D_{4\tau_0}).
	\end{align*}
\end{proof}

\section{CM points with $h(D_{\tau_0})h(D_{4\tau_0})\leqslant 4$}\label{searchcmpoints}

Let $a,b,c\in\Z$ with $a>0, \gcd(a,b,c)=1$. If the quadratic form $f(x,y)=ax^2+bxy+cy^2$ is positive definite, then the discriminant $D=b^2-4ac$ of $f$ satisfies
\[D<0,\quad D\equiv0\text{ or }1\mod 4.\]
We can factorize $f$ as the form
\[f(x,y)=a(x-\tau_0y)(x-\bar{\tau}_0y)\]
for a unique CM point $\tau_0\in\mathcal H.$ It can be shown that $f$ is a reduced form (i.e. $|b|\leqslant a\leqslant c$, and $b\geqslant 0$ if either $|b|=a$ or $a=c$) if and only if $\tau_0$ lies in the fundamental domain
\[\mathcal{F}=\{\tau\in\mathcal H\mid|\re(\tau)|\leqslant 1/2,|\tau|\geqslant 1, \text{and }\re(\tau)\geqslant 0\text{ if }|\re(\tau)|= 1/2\text{ or }|\tau|= 1\}\]
of $\SL(2,\Z)$. It is well known that there are exactly $h(D)$ CM points of discriminant $D$ in $\mathcal F$, each corresponds to a reduced quadratic form with discriminant $D$.

Recall that a discriminant $D<0$ is said to be \textit{fundamental} if there is an imaginary quadratic field $K=\Q(\sqrt{-N})$ with $N\geqslant 1$ and square-free, such that $D=d_K$, where
\[d_K=\begin{cases}
	N, &\text{if}\  N\equiv 1\mod 4,\\
	4N, &\text{otherwise}
\end{cases}\]
is the field discriminant of $K$. It is a famous result that Gauss conjectured and K. Heegner and H. M. Stark proved in \cite{Heegner}  and \cite{Stark} that the only negative fundamental discriminants $d_K$ with $h(d_K)=1$ are
\begin{equation}\label{classnumber1}
	d_K=-3,-4,-7,-8,-11,-19,-43,-67,-163.
\end{equation}
Later, Stark also determined all fundamental discriminants of class number $2$ in \cite{Stark2}, they are
\begin{align}\label{classnumber2}
	d_K=&-15,-20,-24,-35,-40,-51,-52,-88,-91,-115,-123,-148, \\
	&-187,-232,-235,-267,-403,-427.\notag
\end{align}

Recall that an \textit{order} $\mathcal O$ in a quadratic field $K$ is a $\Z$-submodule of the form 
\[\mathcal O=\Z+f\mathcal O_K,\]
where $\mathcal O_K$ is the ring of integers of $K$ and $f=[\mathcal O_K:\mathcal O]$ is the \textit{conductor} of $\mathcal O$. In particular, $\mathcal O_K$ is the maximal order of conductor $1$. The discriminant of $\mathcal O$ is defined by $f^2d_K$, and we have \cite[Corollary 7.28]{Cox}
\begin{equation}\label{orderformula}
	h(f^2d_K)=\frac{h(d_K)f}{[\mathcal{O}_K^*:\mathcal{O}^*]}\prod_{p\mid f}\left(1-\left(\frac{d_K}{p}\right)\frac{1}{p}\right),
\end{equation}

By using \eqref{orderformula}, together with \eqref{classnumber1} and \eqref{classnumber2}, we can determine all the discriminants $D$ such that $h(D)\leqslant 2$. To achieve this, consider the unit groups of imaginary quadratic fields $K=\Q(\sqrt{-N})$
\[\mathcal O^*_K=\begin{cases}
	\{\pm 1, \pm i\}, &\text{if}\  N= -1,\\
	\{\pm 1,\frac{1}{2}(\pm 1\pm i\sqrt{3})\}, &\text{if}\  N= -3,\\
	\{\pm 1\}, &\text{otherwise.}
\end{cases}\]
First, suppose that $\mathcal O^*_K=\{\pm 1\}$. If $f>6$, then
\[f\prod_{p\mid f}\left(1-\left(\frac{d_K}{p}\right)\frac{1}{p}\right)>2,\]
and \eqref{orderformula} tells us that the only possible $(d_K,f)$ such that
\[\begin{cases}
	d_K\neq -3\text{ or }-4,\\
	f>1,\\
	h(f^2d_K)\leqslant 2,
\end{cases}\]
are
\[(-7,2),(-8,2),(-15,2),(-8,3),(-11,3),(-7,4).\]
Similarly, when $d_k=-3\text{ or }-4$, the only possible $(d_K,f)$ such that $f>1$ and $h(f^2d_K)\leqslant 2$ are
\[(-3,2),(-4,2),(-3,3),(-4,3),(-3,4),(-4,4),(-3,5),(-4,5),(-3,7).\]
This completes our goal.
\begin{thm}\label{classnumber<=2}
	Let $D\equiv0\text{ or }1\mod 4$ be negative. Then
	\begin{enumerate}
		\item $h(D)=1$ if and only if $D$ is one of the following numbers:
		\[-3,-4,-7,-8,-11,-12,-16,-19,-27,-28,-43,-67,-163.\]
		\item $h(D)=2$ if and only if $D$ is one of the following numbers:
		\begin{align*}
			&-15,-20,-24,-32,-35,-36,-40,-48,-51,-52,-60,-64,-72,-75,-88,-91,\\
			&-99,-100,-112,-115,-123,-147,-148,-187,-232,-235,-267,-403,-427.
		\end{align*}
	\end{enumerate}
\end{thm}

Let $S,T$ be the generators of $\SL(2,\Z)$ as in Section \ref{CNADOL}. Note that
\begin{equation}\label{RIGHTCOSETS}
	\mathcal F \cup S\mathcal F\cup ST\mathcal F\cup ST^{-1}\mathcal F\cup ST^2\mathcal F \cup ST^{-2}\mathcal F\cup ST^2S\mathcal F\cup ST^{-2}S\mathcal F
\end{equation}
covers the region $\mathcal F'$ in Theorem \ref{mahlermeasureaslambda}; see Figure \ref{picofcovering}. We can now apply Algorithm \ref{Searchcmpoints} to get a list of CM points $\tau_0\in\mathcal F'$ such that $h(D_{\tau_0})h(D_{4\tau_0})\leqslant 4$ and have distinct values of $\lambda(2\tau_0)$ (by numerical calculation).

\begin{algo}\label{Searchcmpoints}
	\ 

	\textsc{Input}: The discriminants in Theorem \ref{classnumber<=2}.
	
	\textsc{Output}: A list (Table \ref{cmpoints}) of CM points in $\mathcal F'$ with $h(D_{\tau_0})h(D_{4\tau_0})\leqslant 4$.
	\begin{enumerate}
		\item For each $D$ in Theorem \ref{classnumber<=2}, determine all CM points of discriminant $D$ in the fundamental domain $\mathcal F$ of $\SL(2,\Z)$.
		\item Determine a full list of CM points that lie in $\mathcal F'$ with $h(D_{\tau0})\leqslant 2$ by using $S,ST,ST^{-1},$ $ST^2,ST^{-2},ST^2S,ST^{-2}S$ (they come from the covering \eqref{RIGHTCOSETS}) to translate the CM points obtained by (1).
		\item For each point $\tau_0$ obtained by (2). If $h(D_{\tau_0})h(D_{4\tau_0})\leqslant 4$, then calculate $\lambda(2\tau_0)$ numerically and take a CM point for each different value to make Table \ref{cmpoints}.
	\end{enumerate}
\end{algo}

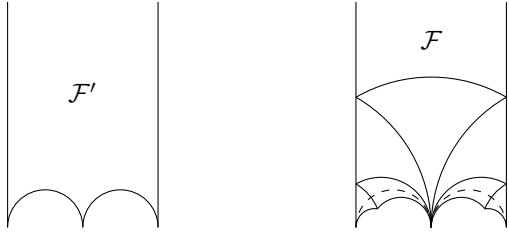
\begin{figure}[h]
	\centering
	\subfigure{\begin{tikzpicture}
			\draw (-1,0) -- (-1,3);
			\draw (1,0) -- (1,3);
			\draw (0,0) arc(0:180:0.5);
			\draw (0,0) arc(180:0:0.5);
			\node at (90:1.8) {$\mathcal F'$};
	\end{tikzpicture}}\hspace{2.5cm}
	\subfigure{\begin{tikzpicture}
			\draw (-1,0) -- (-1,3);
			\draw (1,0) -- (1,3);
			\draw[dashed] (0,0) arc(0:180:0.5);
			\draw[dashed] (0,0) arc(180:0:0.5);
			\draw (1,1.732) arc(60:120:2);
			\node at (90:2.5) {$\mathcal F$};
			\draw (0,0) arc(0:60:2);
			\draw (0,0) arc(180:120:2);
			\draw (0,0) arc(0:120:0.6666);
			\draw (0,0) arc(180:60:0.6666);
			\draw (0,0) arc(180:38.2132:0.4);
			\draw (0,0) arc(0:141.787:0.4);
			\draw (1,0.5773) arc(120:158:0.6666);
			\draw (-1,0.5773) arc(60:22:0.6666);
			\draw (1,0) arc(0:99:0.25);
			\draw (-1,0) arc(180:81:0.25);
	\end{tikzpicture}}
	\caption{The region $\mathcal{F}'$ in Theorem \ref{mahlermeasureaslambda} and its covering \eqref{RIGHTCOSETS}.}\label{picofcovering}
\end{figure}

\begin{table}
	\caption{The outputs of Algorithm \ref{Searchcmpoints}}\label{cmpoints}
	\scalebox{0.95}{
	\begin{tabular}{lccl}
		\hline
		\makebox[3cm][l]{\gape{\ \ \ \ $\tau_0$}} & \makebox[3cm][c]{\gape{$h(D_{2\tau_0})$}} & \makebox[3cm][c]{\gape{$h(D_{\tau_0})h(D_{4\tau_0})$}} & \makebox[6cm][c]{\gape{$\lambda(2\tau_0)$ (take $5$ valid digits)}} \\
		\hline
		$(2,-2,1)$ & $1$ & $1$ & \hspace{1.2cm}$-1.0000$ \\
		$(4,0,1)$ & $1$ & $1$ & \hspace{1.2cm}$0.50000$ \\
		$(8,-4,1)$ & $1$ & $1$ & \hspace{1.2cm}$2.0000$ \\
		$(16,16,5)$ & $1$ & $2$ & \hspace{1.2cm}$-32.970$ \\
		$(16,0,1)$ & $1$ & $2$ & \hspace{1.2cm}$0.97056$ \\
		$(1,0,1)$ & $1$ & $2$ & \hspace{1.2cm}$0.029437$ \\
		$(5,-4,1)$ & $1$ & $2$ & \hspace{1.2cm}$33.970$ \\
		$(4,-4,5)$ & $1$ & $4$ & \hspace{1.2cm}$-0.030330$ \\
		$(20,-4,1)$ & $1$ & $4$ & \hspace{1.2cm}$1.03033$ \\
		$(8,8,3)$ & $1$ & $2$ & \hspace{1.2cm}$-4.8284$ \\
		$(8,0,1)$ & $1$ & $2$ & \hspace{1.2cm}$0.82842$ \\
		$(2,0,1)$ & $1$ & $2$ & \hspace{1.2cm}$0.17157$ \\
		$(6,4,1)$ & $1$ & $2$ & \hspace{1.2cm}$5.8284$ \\
		$(4,4,3)$ & $1$ & $4$ & \hspace{1.2cm}$-0.20710$ \\
		$(12,4,1)$ & $1$ & $4$ & \hspace{1.2cm}$1.20710$ \\
		$(3,3,1)$ & $1$ & $2$ & \hspace{1.2cm}$-13.928$ \\
		$(1,1,1)$ & $1$ & $2$ & \hspace{1.2cm}$-0.071796$ \\
		$(16,4,1)$ & $1$ & $2$ & \hspace{1.2cm}$1.07179$ \\
		$(16,12,3)$ & $1$ & $2$ & \hspace{1.2cm}$14.928$ \\
		$(4,0,3)$ & $1$ & $4$ & \hspace{1.2cm}$0.066987$ \\
		$(12,0,1)$ & $1$ & $4$ & \hspace{1.2cm}$0.93301$ \\
		$(4,2,1)$ & $1$ & $1$ & \hspace{1.2cm}$0.50000-0.86602 i$ \\
		$(4,-2,1)$ & $1$ & $1$ & \hspace{1.2cm}$0.50000+0.86602 i$ \\
		$(7,7,2)$ & $1$ & $2$ & \hspace{1.2cm}$-253.99$ \\
		$(1,1,2)$ & $1$ & $2$ & \hspace{1.2cm}$-0.0039370$ \\
		$(32,4,1)$ & $1$ & $2$ & \hspace{1.2cm}$1.0039$ \\
		$(32,28,7)$ & $1$ & $2$ & \hspace{1.2cm}$254.99$ \\
		$(4,0,7)$ & $1$ & $4$ & \hspace{1.2cm}$0.0039216$ \\
		$(28,0,1)$ & $1$ & $4$ & \hspace{1.2cm}$0.99607$ \\
		$(2,1,1)$ & $1$ & $1$ & \hspace{1.2cm}$0.031250-0.24803 i$ \\
		$(2,-1,1)$ & $1$ & $1$ & \hspace{1.2cm}$0.031250+0.24803 i$ \\
		$(4,3,1)$ & $1$ & $1$ & \hspace{1.2cm}$0.50000-3.9686 i$ \\
		$(4,-3,1)$ & $1$ & $1$ & \hspace{1.2cm}$0.50000+3.9686 i$ \\
		$(8,2,1)$ & $1$ & $1$ & \hspace{1.2cm}$0.96875-0.24803 i$ \\
		$(8,-2,1)$ & $1$ & $1$ & \hspace{1.2cm}$0.96875+0.24803 i$ \\
		$(4,-1,1)$ & $2$ & $4$ & \hspace{1.2cm}$0.50000+0.30096 i$ \\
		$(4,1,1)$ & $2$ & $4$ & \hspace{1.2cm}$0.50000-0.30096 i$ \\
		$(8,7,2)$ & $2$ & $4$ & \hspace{1.2cm}$0.50000-27.411 i$ \\
		$(8,-7,2)$ & $2$ & $4$ & \hspace{1.2cm}$0.50000+27.411 i$ \\
		$(2,1,2)$ & $2$ & $4$ & \hspace{1.2cm}$0.00066519-0.036468 i$ \\
		$(2,-1,2)$ & $2$ & $4$ & \hspace{1.2cm}$0.00066519+0.036468 i$ \\
		$(6,3,1)$ & $2$ & $4$ & \hspace{1.2cm}$1.4680-0.88368 i$ \\
		$(6,-3,1)$ & $2$ & $4$ & \hspace{1.2cm}$1.4680+0.88368 i$ \\
		$(8,6,3)$ & $2$ & $4$ & \hspace{1.2cm}$-0.46808-0.88368 i$ \\
		$(8,-6,3)$ & $2$ & $4$ & \hspace{1.2cm}$-0.46808+0.88368 i$ \\
		$(16,2,1)$ & $2$ & $4$ & \hspace{1.2cm}$0.99933-0.036468 i$ \\
		$(16,-2,1)$ & $2$ & $4$ & \hspace{1.2cm}$0.99933+0.036468 i$ \\
		\hline
	\end{tabular}}
\end{table}

\section{Mahler Measures and $L$-values of modular forms}\label{tableofmahlermeasure}

To illustrate how the data in Table \ref{cmpoints} allows us to write Mahler measures as $L$-values of modular forms, we calculate two examples in detail, they are $m\left(12\pm 8\sqrt{2}\right)$.


Before working out these two examples, recall that the modular Weber functions (see, for example \cite[\S 12]{Cox}) are 

\[\mathfrak f(\tau)=q^{\frac{1}{24}}\frac{\eta((\tau+1)/2)}{\eta(\tau)},\ \ \mathfrak f_1(\tau)=\frac{\eta(\tau/2)}{\eta(\tau)},\ \ \mathfrak{f}_2(\tau)=\sqrt{2}\frac{\eta(2\tau)}{\eta(\tau)}.\]
They satisfy the following identities:
\begin{align}
	&\mathfrak f(\tau)\mathfrak f_1(\tau)\mathfrak f_2(\tau)=\sqrt{2},\ \ \mathfrak f_1(2\tau)\mathfrak f_2(\tau)=\sqrt{2},\nonumber\\[6pt]\label{jandWeber}
	&j(\tau)=\frac{(\mathfrak f(\tau)^{24}-16)^3}{\mathfrak f(\tau)^{24}}=\frac{(\mathfrak f_1(\tau)^{24}+16)^3}{\mathfrak f_1(\tau)^{24}}=\frac{(\mathfrak f_2(\tau)^{24}+16)^3}{\mathfrak f_2(\tau)^{24}},\\[4pt]\label{jandWeber2}
	&\lambda(2\tau)=\frac{\mathfrak f_1(2\tau)^8\mathfrak f_2(2\tau)^{16}}{16}.
\end{align}

\begin{proof}[Proof of the cases $k=12\pm8\sqrt{2}$ in Theorem \ref{mainresults1}]
	Take $\tau_0=(1,0,1)=i$. Since $D(2\tau_0)=-16$ and $h(-16)=1$, we know that $j(2\tau_0)$ is a rational integer. It is not difficult to find that
	\[j(2\tau_0)=287496\]
	by numerical calculation. Thus, by \eqref{jandWeber} and some numerical calculation for $\mathfrak f_1$ and $\mathfrak f_2$, we have
	\[\mathfrak f_1(2\tau_0)^{24}=512,\ \ \mathfrak f_2(2\tau_0)^{24}=-280+192\sqrt{2}.\]
	Then \eqref{jandWeber2} tells us that
	\[\lambda(2\tau_0)=17-12\sqrt{2},\]
	and the corresponding 
	\[k=\frac{4}{\sqrt{\lambda(2\tau_0)}}=12+8\sqrt{2}.\]
	Now, all the ingredients are in place to use Theorem \ref{mahlermeasureaslambda}. Take $\tau=\tau_0$ in \eqref{mahlermeasureandlambda}, we have
	\begin{align*}
		m\left(\!12+8\sqrt{2}\right)&=\re\left(\frac{16\im(\tau_0)}{\pi^2}\underset{m,n\in\Z}{{\sum}'}\frac{\chi_{-4}(n)}{(4m\tau_0+n)^2(4m\bar\tau_0+n)}\right)\\
		&=\re\left(\frac{16}{\pi^2}\underset{m,n\in\Z}{{\sum}'}\frac{\chi_{-4}(n)(4m\bar\tau_0+n)}{|4m\tau_0+n|^4}\right)\\
		&=\frac{16}{\pi^2}\underset{m,n\in\Z}{{\sum}'}\frac{\chi_{-4}(n)n}{\left(16m^2+n^2\right)^2}\\
		&=\frac{32}{\pi^2}\left(\frac{1}{2}\underset{m,n\in\Z}{{\sum}'}\frac{\chi_{-4}(n)n}{\left(16m^2+n^2\right)^2}\right).
	\end{align*}
	One can deduce from \cite[Corollary 14.3.16]{Mfaca} that the theta function
	\[f_{12+8\sqrt{2}}=\frac{1}{2}\underset{m,n\in\Z}{{\sum}}\chi_{-4}(n)nq^{16m^2+n^2}=q\!-\!3q^9\!+\!2q^{17}\!-\!q^{25}\!+\!10q^{41}\!-\!7q^{49}\!+\!\cdots\]
	is a normalized cusp form in $\mathcal{S}_2(\Gamma_0(64))$. This implies that
	\[m\left(12+8\sqrt{2}\right)=\frac{32}{\pi^2}L(f_{12+8\sqrt{2}},2).\]

	Next, Take $\tau_1=(5,-4,1)=\frac{2+i}{5}$. We can calculate that
	\begin{align*}
		&j(2\tau_1)=287496,\ \ f_1(2\tau_1)^{24}=512,\ \ \mathfrak f_2(2\tau_1)^{24}=-280-192\sqrt{2},\\[4pt]
		&\lambda(2\tau_1)=17+12\sqrt{2},\ \ k=\frac{4}{\sqrt{\lambda(2\tau_1)}}=12-8\sqrt{2}.
	\end{align*}
	By taking $\tau=\tau_1$ in \eqref{mahlermeasureandlambda}, we have
	\begin{align*}
		m\left(\!12-8\sqrt{2}\right)&=\frac{16}{5\pi^2}\underset{m,n\in\Z}{{\sum}'}\frac{\chi_{-4}(n)\left(\frac{8m}{5}+n\right)}{\left(\left(\frac{8m}{5}+n\right)^2+\left(\frac{4m}{5}\right)^2\right)^2}\\
		&=\frac{64}{\pi^2}\left(\frac{1}{4}\underset{m,n\in\Z}{{\sum}'}\frac{\chi_{-4}(n)(8m+5n)}{\left(16m^2+16mn+5n^2\right)^2}\right).
	\end{align*}
	Once again, according to \cite[Corollary 14.3.16]{Mfaca}, the theta function
	\begin{align*}
		f_{12-8\sqrt{2}}&=\frac{1}{4}\underset{m,n\in\Z}{{\sum}}\chi_{-4}(n)(8m+5n)q^{16m^2+16mn+5n^2}\\
		&= q^5-3q^{13}+5q^{29}+q^{37}-3q^{45}-7q^{53}+5q^{61}+2q^{85}+\cdots
	\end{align*}
	is a normalized cusp form in $\mathcal{S}_2(\Gamma_0(64))$. This implies that
	\[m\left(12-8\sqrt{2}\right)=\frac{64}{\pi^2}L(f_{12-8\sqrt{2}},2).\]
\end{proof}

For the other CM points in Table \ref{cmpoints} with $h(D_{2\tau_0})=1$, one can also repeat the above calculation to obtain identities of the form
\[m(k)=\frac{c_k}{\pi^2}L(f_k,2),\]
where $f_k$ are normalized cusp forms in $\mathcal{S}_2(\Gamma_0(N_k))$. To avoid lengthy calculations, we list the results directly in Table \ref{mahlermeasureandmodularforms}.

\begin{remark}\ 
	\begin{enumerate}
		\item As mentioned in the Introduction, the first three cases in table \textnormal{\ref{mahlermeasureandmodularforms}} are the results \eqref{Vliegasexamples} proved by Villegas. The cases $k=4\pm 4\sqrt{2}$ have also been proven in \cite{Guo}. And the cases $k=\sqrt[4]{8}(\sqrt{2}-1)i$ and $\sqrt[4]{8}(\sqrt{2}+1)$ are in fact proved by Samart in \cite{Sam15}.
		\item 	We do not guarantee that we have found all $k=\frac{4}{\sqrt{\lambda(2\tau_0)}}$ with degree $\leqslant 4$ that come from CM points. Since the degree of $\lambda(2\tau_0)$ may be strictly less than $h(D_{\tau_0})h(D_{4\tau_0}).$
	\end{enumerate}

\end{remark}

\begin{table}
	\caption{Mahler measures that comes from CM points in Table \ref{cmpoints}}\label{mahlermeasureandmodularforms}
	\scalebox{0.73}{
		\begin{tabular}{cccl}
			\hline
			\makebox[4cm][c]{\gape{$k$}} & \makebox[4cm][c]{\gape{$c_k$}} & \makebox[3cm][c]{\gape{$N_k$}} & \makebox[8cm][c]{\gape{$f_k$}}\\
			\hline\\[-9pt]
			$4 i$ & $16$ & $32$  & \hspace{1.2cm}$\frac{1}{2}\underset{m,n\in\Z}{{\sum}}\chi_{-4}(n)(2m+n)q^{8m^2+4mn+n^2}$\\
			$4\sqrt{2}$ & $16$ & $64$  & \hspace{1.2cm}$\frac{1}{2}\underset{m,n\in\Z}{{\sum}}\chi_{-4}(n)nq^{4m^2+n^2}$\\
			$2\sqrt{2}$ & $8$ & $32$  & \hspace{1.2cm}$\frac{1}{2}\underset{m,n\in\Z}{{\sum}}\chi_{-4}(n)(m+n)q^{2m^2+2mn+n^2}$\\
			$\sqrt[4]{8}(\sqrt{2}-1) i$ & $8$ & $64$  & \hspace{1.2cm}$\frac{1}{2}\underset{m,n\in\Z}{{\sum}}\chi_{-4}(n)(-2m+n)q^{5m^2-4mn+n^2}$\\
			$\sqrt[4]{8}(\sqrt{2}+1)$ & $8$ & $64$  & \hspace{1.2cm}$\frac{1}{2}\underset{m,n\in\Z}{{\sum}}\chi_{-4}(n)nq^{m^2+n^2}$\\
			$12+8\sqrt{2}$ & $32$ & $64$  & \hspace{1.2cm}$\frac{1}{2}\underset{m,n\in\Z}{{\sum}}\chi_{-4}(n)nq^{16m^2+n^2}$\\
			$12-8\sqrt{2}$ & $64$ & $64$  & \hspace{1.2cm}$\frac{1}{4}\underset{m,n\in\Z}{{\sum}}\chi_{-4}(n)(8m+5n)q^{16m^2+16mn+5n^2}$\\
			$8 i\sqrt{3\sqrt{2}+4}$ & $32$ & $256$  & \hspace{1.2cm}$\frac{1}{2}\underset{m,n\in\Z}{{\sum}}\chi_{-4}(n)(2m+n)q^{20m^2+4mn+n^2}$\\
			$8\sqrt{3\sqrt{2}-4}$ & $256$ & $256$  & \hspace{1.2cm}$\frac{1}{16}\underset{m,n\in\Z}{{\sum}}\chi_{-4}(n)(2m+5n)q^{4m^2+4mn+5n^2}$\\
			$2i\sqrt{2\sqrt{2}-2}$ & $8\sqrt{2}$ & $64$  & \hspace{1.2cm}$\frac{1}{2}\underset{m,n\in\Z}{{\sum}}\chi_{-4}(n)(-2m+n)q^{6m^2-4mn+n^2}$\\
			$2i\sqrt{2\sqrt{2}+2}$ & $8\sqrt {2}$ & $64$  & \hspace{1.2cm}$\frac{1}{2}\underset{m,n\in\Z}{{\sum}}\chi_{-4}(n)nq^{2m^2+n^2}$\\
			$4+4\sqrt{2}$ & $16\sqrt {2}$ & $64$  & \hspace{1.2cm}$\frac{1}{2}\underset{m,n\in\Z}{{\sum}}\chi_{-4}(n)nq^{8m^2+n^2}$\\
			$4-4\sqrt{2}$ & $32\sqrt {2}$ & $64$  & \hspace{1.2cm}$-\frac{1}{4}\underset{m,n\in\Z}{{\sum}}\chi_{-4}(n)(4m-3n)q^{8m^2-8mn+3n^2}$\\
			$4 i\sqrt{2\sqrt{2}+2}$ & $16\sqrt{2}$ & $128$  & \hspace{1.2cm}$\frac{1}{2}\underset{m,n\in\Z}{{\sum}}\chi_{-4}(n)(-2m+n)q^{12m^2-4mn+n^2}$\\
			$4\sqrt{2\sqrt{2}-2}$ & $64\sqrt{2}$ & $128$  & \hspace{1.2cm}$-\frac{1}{8}\underset{m,n\in\Z}{{\sum}}\chi_{-4}(n)(2m-3n)q^{4m^2-4mn+3n^2}$\\
			$(8-4\sqrt{3}) i$ & $48\sqrt{3}$ & $48$  & \hspace{1.2cm}$-\frac{1}{2}\underset{m,n\in\Z}{{\sum}}\chi_{-4}(n)(2m-n)q^{16m^2-12mn+3n^2}$\\
			$(8+4\sqrt{3}) i$ & $16\sqrt{3}$ & $48$  & \hspace{1.2cm}$-\frac{1}{2}\underset{m,n\in\Z}{{\sum}}\chi_{-4}(n)(2m-n)q^{16m^2-4mn+n^2}$\\
			$\sqrt{2}+\sqrt{6}$ & $6\sqrt{3}$ & $48$  & \hspace{1.2cm}$-\frac{1}{6}\underset{m,n\in\Z}{{\sum}}\chi_{-4}(n)(m-2n)q^{m^2-mn+n^2}$\\
			$\sqrt{2}-\sqrt{6}$ & $2\sqrt{3}$ & $48$  & \hspace{1.2cm}$-\frac{1}{2}\underset{m,n\in\Z}{{\sum}}\chi_{-4}(n)(3m-2n)q^{3m^2-3mn+n^2}$\\
			$4\sqrt{2}+4\sqrt{6}$ & $16\sqrt{3}$ & $192$  & \hspace{1.2cm}$\frac{1}{2}\underset{m,n\in\Z}{{\sum}}\chi_{-4}(n)nq^{12m^2+n^2}$\\
			$4\sqrt{2}-4\sqrt{6}$ & $48\sqrt{3}$ & $192$  & \hspace{1.2cm}$\frac{1}{2}\underset{m,n\in\Z}{{\sum}}\chi_{-4}(n)nq^{4m^2+3n^2}$\\
			$2\sqrt{3}\pm 2i$ & $8\sqrt{3}$ & $48$  & \hspace{1.2cm}$-\frac{1}{2}\underset{m,n\in\Z}{{\sum}}\chi_{-4}(n)(m-n)q^{4m^2-2mn+n^2}$\\
			$(32-12\sqrt{7}) i$ & $112\sqrt{7}$ & $112$  & \hspace{1.2cm}$-\frac{1}{2}\underset{m,n\in\Z}{{\sum}}\chi_{-4}(n)(2m-n)q^{32m^2-28mn+7n^2}$\\
			$(32+12\sqrt{7}) i$ & $16\sqrt{7}$ & $112$  & \hspace{1.2cm}$-\frac{1}{2}\underset{m,n\in\Z}{{\sum}}\chi_{-4}(n)(2m-n)q^{32m^2-4mn+n^2}$\\
			$\frac{3\sqrt{2}}{2}+\frac{\sqrt{14}}{2}$ & $14\sqrt{7}$ & $112$  & \hspace{1.2cm}$-\frac{1}{14}\underset{m,n\in\Z}{{\sum}}\chi_{-4}(n)(m-4n)q^{m^2-mn+2n^2}$\\
			$\frac{3\sqrt{2}}{2}-\frac{\sqrt{14}}{2}$ & $2\sqrt{7}$ & $112$  & \hspace{1.2cm}$-\frac{1}{2}\underset{m,n\in\Z}{{\sum}}\chi_{-4}(n)(7m-4n)q^{7m^2-7mn+2n^2}$\\
			$24\sqrt{2}+8\sqrt{14}$ & $16\sqrt{7}$ & $448$  & \hspace{1.2cm}$\frac{1}{2}\underset{m,n\in\Z}{{\sum}}\chi_{-4}(n)nq^{28m^2+n^2}$\\
			$24\sqrt{2}-8\sqrt{14}$ & $112\sqrt{7}$ & $448$  & \hspace{1.2cm}$\frac{1}{2}\underset{m,n\in\Z}{{\sum}}\chi_{-4}(n)nq^{4m^2+7n^2}$\\
			$6\pm2i\sqrt{7}$ & $8\sqrt{7}$ & $56$  & \hspace{1.2cm}$-\frac{1}{2}\underset{m,n\in\Z}{{\sum}}\chi_{-4}(n)(m-n)q^{8m^2-2mn+n^2}$\\
			$\frac{3}{2}\pm i\frac{\sqrt{7}}{2}$ & $4\sqrt{7}$ & $28$  & \hspace{1.2cm}$-\frac{1}{4}\underset{m,n\in\Z}{{\sum}}\chi_{-4}(n)(3m-2n)q^{4m^2-3mn+n^2}$\\
			$\frac{3\sqrt{7}}{2}\pm\frac{i}{2}$ & $4\sqrt{7}$ & $56$  & \hspace{1.2cm}$-\frac{1}{4}\underset{m,n\in\Z}{{\sum}}\chi_{-4}(n)(m-2n)q^{2m^2-mn+n^2}$\\
			\hline
	\end{tabular}}
\end{table}

\section{Beilinson's Conjecture for curves over number fields}\label{Beilinson}

This section is devoted to give a relatively detailed statement of Beilinson's conjecture for curves over number field. The content of \cite{DJZ06} is widely quoted.


We start with some $K$-theory settings. Let $F$ be a field, the group $K_2(F)$ can be defined as
\[F^*\otimes_{\Z} F^*/\langle a\otimes(1-a),a\in F,a\neq 0,1\rangle.\]
The class of $a\otimes b$ is usually denoted by $\{a,b\}$. One can see form definition that $K_2(F)$ is an abelian group (we use ``+" for the operator) given by generators $\{a,b\}$ for $a,b\in F^*$ and relations
\begin{align*}
	\{a_1a_2,b\}&=\{a_1,b\}+\{a_2,b\},\\
	\{a,b_1b_2\}&=\{a,b_1\}+\{a,b_2\},\\
	\{a,1-a\}&=0,\ \text{for}\ a\in F, a\neq 0,1.
\end{align*}

Let $K$ be a number field with the ring of
integers $\mathcal O_K$ and $C/K$ be a (non-singular, projective, geometrically irreducible) curve of genus $g$ with function field $F=K(C)$. Define the tame $K_2$ of $C$ as
\[K_2^T(C)=\ker\left(K_2(F)\overset{T}{\longrightarrow}\bigoplus_{x\in C(\overline{\Q})}\overline{\Q}^*\right),\]
where the $x$-component of $T$ is given by the \textit{tame symbol}
\begin{equation}\label{tamesymbol}
	T_x(\{a,b\})=(-1)^{\mathrm{ord}_x(a)\mathrm{ord}_x(b)}\frac{a^{\mathrm{ord}_x(b)}}{b^{\mathrm{ord}_x(a)}}(x).
\end{equation}

For each $\{a,b\}\in K_2^T(C)$, there is an almost everywhere defined $1$-form
\[\eta(a,b)=\log|a|d\arg b-\log|b|d\arg a\]
on the Riemann surface $C(\CC)$, where $d\arg a$ is defined by $\Im(da/a)$. One can check that
\begin{align*}
	\langle\cdot,\cdot\rangle: H_1(C(\CC);\Z)&\times K_2^T(C)/\text{torsion}\to\R,\\
	(\gamma,\{a,b\})&\mapsto\frac{1}{2\pi}\int_\gamma\eta(a,b).
\end{align*}
is a well-defined pairing. (Note that when calculating the above integral, we choose a representative of $\gamma$ that avoids the set of zeros and poles of $a$ and $b$.) Since $\eta(a,b)$ changes sign under complex conjugation, it can be shown that $\langle\cdot,\cdot\rangle$ is in fact equal to zero on $H_1(C(\CC);\Z)^+$, the subgroup of $H_1(C(\CC);\Z)$ consisting of  loops that invariant under complex conjugation. We only need to consider the above pairing for $\gamma\in H_1(C(\CC);\Z)^-$, the subgroup of $H_1(C(\CC);\Z)$ on which complex conjugation acts by multiplying $-1$. Thus, we obtain the \textit{regulator pairing}
\begin{equation}\label{regulatorpairing}
	\langle\cdot,\cdot\rangle: H_1(C(\CC);\Z)^-\times K_2^T(C)/\text{torsion}\to\R.
\end{equation}

From basic topology we know that $H_1(C(\CC);\Z)$ has rank $g$. Beilinson originally conjectured that $K_2^T(C)/\text{torsion}$ also has rank $g$ and the pairing \eqref{regulatorpairing} is non-degenerate. However, Bloch
and Grayson discovered in \cite{BG86} that the rank of $K_2^T(C)/\text{torsion}$ can sometimes $>g$. As a modification, we turn to a subgroup $K_2(C;\Z)$ of $K_2^T(C)/\text{torsion}$ defined by
\[K_2(C;\Z)=\frac{\ker(K_2(F)\overset{T}{\to}\bigoplus_{\mathcal D}\mathbb{F}(\mathcal D)^*)}{\text{torsion}}\subseteq\frac{K_2^T(C)}{\text{torsion}},\]
where $\mathcal D$ runs through all irreducible curves on a regular proper model $\mathcal C/\mathcal O_K$ of $C/K$, and $\mathbb F(\mathcal D)$ is the residue field at $\mathcal D$. The
$\mathcal D$-component of $T$ is given by the tame symbol similar to \eqref{tamesymbol}:
\[T_{\mathcal D}(\{a,b\})=(-1)^{v_{\mathcal D}(a)v_{\mathcal D}(b)}\frac{a^{v_{\mathcal D}(b)}}{b^{v_{\mathcal D}(a)}}(\mathcal D),\]
where $v_{\mathcal D}$ is is the valuation corresponding to $\mathcal D$. It is expected that $K_2(C;\Z)$ has rank $g$.

For each embedding $\sigma:K\to \CC$, by applying $\sigma$ to coefficients of the equation defining $C$, we get a curve $C^\sigma$. Let $X^\sigma$ be the connected Riemann surface of genus $g$ associated to $C^\sigma(\CC)$ and let $X$ be the disjoint union of all $X^\sigma$. In fact, $X$ is exactly the Riemann surface  corresponding to $C\times_\Q\CC$. Through the action on $\CC$ in $C\times_\Q\CC$, the complex conjugation
acts on $X$, then on $H_1(X;\Z)$. Similarly, one can define $H_1(X;\Z)^-$ to be the subgroup of $H_1(X;\Z)$ on which complex conjugation acts by multiplying  $-1$. We have
\[H_1(X;\Z)^-=\bigoplus_\sigma H_1(X^\sigma;\Z)^-\]
and thus the rank of $H_1(X;\Z)^-$ is $r=g[K:\Q]$. 

To state Beilinson's conjecture for curves over number fields, we extend the definition of the regulator pairing $\langle\cdot,\cdot\rangle$ to all $H_1(X;\Z)^-$ by putting
\[\langle\gamma,\{a,b\}\rangle=\frac{1}{2\pi}\int_\gamma\eta(a^\sigma,b^\sigma)\]
for $\gamma\in H_1(X^\sigma;\Z)^-$, where $a^\sigma$ is the function on $X^\sigma$ obtained by applying $\sigma$ to the coefficients of $a\in K(C)$.

\begin{conjecture}[Beilinson]\label{Beilinsonsconjecture}
	Let $C$ be a non-singular, projective, geometrically irreducible curve of genus $g$ defined over $K$ and let $X$ be defined as above, then
	\begin{enumerate}
		\item $K_2(C;\Z)$ is a free abelian group of rank $r=g[K:\Q]$ and the pairing $\langle\cdot,\cdot\rangle: H_1(X;\Z)^-\times K_2(C;\Z)\to\R$ is non-degenerate;
		\item The absolute value $R$ (this is called the \textnormal{Beilinson regulator}) of the determinant of the pairing in (1) with respect to $\Z$-bases of $H_1(X;\Z)$ and $K_2(C;\Z)$ is a nonzero rational multiple of $\pi^{-2r}L(C,2)$, where $L(C,s)$ is the $L$-function of $C$.
	\end{enumerate}
\end{conjecture}

The first part of Conjecture \ref{Beilinsonsconjecture} is widely open. In fact, it is not yet known whether $K_2(C;\Z)$ is finite generated, let alone to find its $\Z$-bases. However, if we can construct $r$ elements in $K_2(C;\Z)$ with the associated $R\neq 0,$ then these $r$ elements are
linearly independent over $\Z$. In this case, we may be able to verify the second part of Conjecture \ref{Beilinsonsconjecture}.

\section{Proofs of Theorem \ref{mainresults1}}\label{example12pm8sqrt(2)}

When $k=12\pm8\sqrt{2}$, under the rational transformation \eqref{Weierstrassform}, we have Weierstrass equations
\begin{align*}
	E_{12+8\sqrt{2}}&:Y^2=X^3+\big(66+48\sqrt{2}\big)X^2+X,\\
	E_{12-8\sqrt{2}}&:Y^2=X^3+\big(66-48\sqrt{2}\big)X^2+X.
\end{align*}
They are defined over $K=\Q(\sqrt{2})$. Let $\sigma\in\Gal(K/\Q)$ be the nontrivial element that sends $\sqrt{2}$ to $-\sqrt{2}$. For simplicity, we write $E$ for $E_{12+8\sqrt{2}}$, then $E^{\sigma}$ is for $E_{12-8\sqrt{2}}$ according to notations in Section \ref{Beilinson}.

An elliptic curve over a (Galois) number field is said to be a \textit{$\Q$-curve} if it is isogenic to its every Galois conjugations. Since $E$ has CM and every elliptic curve with CM is a $\Q$-curve (see \cite[Section 2]{BF18}), there must exist an isogeny from $E$ to $E^\sigma$. By using \textsf{PARI/GP}, we  compute the lattices that associated to $E$ and $E^{\sigma}$, and discover that there seems to have an isogeny $\phi:E\to E^\sigma$ of degree $4$ with kernel
\[G=\left\{O,(0,0),\left(-1,4\sqrt{4+3\sqrt{2}}\hspace{1pt}\right),\left(-1,-4\sqrt{4+3\sqrt{2}}\hspace{1pt} \right)\right\}.\]
To obtain this isogeny, we first use Vélu's formula (see \cite[Theorem 12.16]{Was08} or \cite[Theorem 25.1.6]{Gal12}) to construct an elliptic curve
\[\widetilde{E}: Y^2=X^3+(66+48\sqrt{2}) X^2+(1276+960\sqrt{2}) X+137464+96960\sqrt{2}\]
and isogeny $\psi: E\to \widetilde{E}$ with $\ker\psi=G.$ This isogeny  is given by $(X,Y)\mapsto(\psi_1(X),Y\psi_2(X))$, where
\begin{align*}
	\psi_1(X)&=\frac{X^4+2X^3-2\left(127+96\sqrt{2}\right)X^2+2X+1}{X(X+1)^2},\\
	\psi_2(X)&=\frac{(X-1) \left(X^4+4X^3+2 \left(131+96\sqrt{2}\hspace{1pt}\right) X^2+4X+1\right)}{X^2 (X+1)^3}.
\end{align*}
Moreover, we have (by \cite[Theorem 25.1.6]{Gal12})
\begin{equation}\label{differential1}
	{\psi}^*(\omega_{\widetilde E})=\omega_E,
\end{equation}
where $\omega_{\widetilde E}$ and $\omega_{E}$ are invariant differentials of $\widetilde E$ and $E$ defined by $\frac{dX}{2Y}.$ Next, set
$u=\frac{3}{2}-\sqrt{2},r=-\frac{49}{2}+18\sqrt{2},$
one can easily verify that
\begin{align}\label{isomorphism}
	\varphi: \widetilde E&\to E^\sigma,\\
	(X,Y)&\mapsto\left(u^2X+r,u^3Y\right)\nonumber
\end{align}
is an isomorphism over $K$. It follows that $\phi=\varphi\circ\psi:E\to E^\sigma$ is our desire isogeny with kernel $G$. We can explicitly write down its formula as $\phi:(X,Y)\mapsto(\phi_1(X),Y\phi_2(X)),$ where
\begin{align*}
	\phi_1(X)&=\frac{(X-1)^2 \left(\left(17-12 \sqrt{2}\hspace{1pt}\right) X^2-6 \left(5-4\sqrt{2}\hspace{1pt}\right) X+17-12\sqrt{2}\hspace{1pt}\right)}{4X(X+1)^2},\\
	\phi_2(X)&=\frac{\left(99-70 \sqrt{2}\hspace{0.5pt}\right)(X-1) \left(X^4+4 X^3+2\left(131+96 \sqrt{2}\hspace{1pt}\right) X^2+4 X+1\right)}{8 X^2(X+1)^3}.
\end{align*}

By applying $\sigma$ to coefficients of $\phi$, we obtain an isogeny $\phi^\sigma:E^\sigma\to E$. One can check that $\phi^\sigma\circ\phi=[4].$

Let $|x|=1$, we assume $x=e^{ i\theta}$. Then $x+\frac{1}{x}+k=2\cos\theta+k\in\R$ if $k\in \R.$
When $|2\cos\theta+k|>2$, the equation
\begin{equation}\label{quadraticequation}
	y+\frac{1}{y}=-k-x-\frac{1}{x}
\end{equation}
have two real roots $y_1(x), y_2(x)$. They are not equal since the discriminant $\Delta=\left( x+\frac{1}{x}+k\right)^2-4>0$. Note that $y_1(x)y_2(x)=1$, without loss of generality we can always take $y_1(x)$ to be the one that has absolute value greater than $1$. However, when $|2\cos\theta+k|\leqslant 2$, the roots $y_1,y_2$ of \eqref{quadraticequation} are conjugate on the unit circle $S^1=\{x=e^{ i\theta}|\theta\in[-\pi,\pi]\}$. Notice that in this case, $y_1=y_2$ only when $2\cos\theta+k=\pm 2$.

If $k=12+8\sqrt{2}>4$, then for every $\theta\in[-\pi,\pi]$ we have $|2\cos\theta+k|>2$. Let 
\begin{align*}
	y^E_1(x)&=\frac{-x-\frac{1}{x}-12-8\sqrt{2}-\sqrt{\left(x+\frac{1}{x}+12+8 \sqrt{2}\right)^2-4}}{2},\\
	y^E_2(x)&=\frac{-x-\frac{1}{x}-12-8\sqrt{2}+\sqrt{\left(x+\frac{1}{x}+12+8 \sqrt{2}\right)^2-4}}{2}
\end{align*}
be the two roots of \eqref{quadraticequation}. As mentioned in the previous paragraph, $y_1^E$ is taken to be the root that has absolute value greater than 1.
Thus the circle $S^1$ can be lifted to a loop $\gamma_E=\{(x,y_1^E(x))||x|=1\}$ on $E(\CC)$. We take the orientation of $\gamma_E$ induced by $\theta$ increasing. Since complex conjugation reverse the orientation, $\gamma_E$ is an element of $H_1(E(\CC);\Z)^-$. This path is called the \textit{Deninger path}, one can see \cite[\S 7.3]{BZ20} for more details.

However, if $k=12-8\sqrt{2}=0.6862\cdots$, as $x=e^{ i\theta}$ varies on $S^1$, the discriminant $\Delta=0$ when $x=x_0=-5+4\sqrt{2}+2 i \sqrt{10 \sqrt{2}-14}$ or $x=\bar{x}_0$. The parameter $\theta_0\in[-\pi,\pi]$ corresponding to $x_0$ is $\theta_0=\arctan\frac{2\sqrt{2+10\sqrt{2}}}{7}$. Thus \eqref{quadraticequation} has real roots
\begin{align*}
	y^{E^\sigma}_1(x)&=\frac{-x^2-(12-8\sqrt{2})x-1-\sqrt{\left(x^2+(12-8\sqrt{2})x+1\right)^2-4 x^2}}{2 x},\\
	y^{E^\sigma}_2(x)&=\frac{-x^2-(12-8\sqrt{2})x-1+\sqrt{\left(x^2+(12-8\sqrt{2})x+1\right)^2-4 x^2}}{2 x}
\end{align*}
for $\theta\in[-\theta_0,\theta_0]$. We illustrate in Figure \ref{pathofy}
\begin{figure}
	\centering
	\subfigure[The path of $y^{E^\sigma}_1$]{\begin{tikzpicture}
			\draw[very thin,->] (-3.6,0) -- (1,0);
			\draw[very thin,->] (0,-1.2) -- (0,1.2);
			\draw[dashed,-stealth] (-48.9396:1) arc (-48.9396:-75:1);
			\draw[dashed] (-65:1) arc (-65:-175:1);
			\draw[dashed,-stealth] (175:1) arc (175:68:1);
			\draw[dashed] (70:1) arc (70:48.9396:1);
			\draw[very thick] (175:1) -- (-3.21,0.1);
			\draw[very thick] (-175:1) -- (-3.21,-0.08);
			\draw[very thick] (-3.21,0.1) arc(90:270:0.089);
			\fill (48.9396:1) circle (1.5pt);
			\node[above right] at (48.9396:1) {$\pi$};
			\fill (-48.9396:1) circle (1.5pt);
			\node[below right] at (-48.9396:1) {$-\pi$};
			\fill (175:1) circle (1.5pt);
			\fill (-175:1) circle (1.5pt);
			\node[above left] at (177:0.9) {$\theta_0$};
			\node[below left] at (-177:0.9) {$-\theta_0$};
	\end{tikzpicture}}\hspace{1.5cm}
	\subfigure[The path of $y^{E^\sigma}_2$]{\begin{tikzpicture}
			\draw[very thin,->] (-2.0,0) -- (1,0);
			\draw[very thin,->] (0,-1.2) -- (0,1.2);
			\draw[dashed,-stealth] (48.9396:1) arc (48.9396:76:1);
			\draw[dashed] (75:1) arc (75:175:1);
			\draw[very thick] (175:1) -- (-0.4,0.1);
			\draw[very thick] (-175:1) -- (-0.4,-0.08);
			\draw[very thick] (-0.4,0.1) arc(90:-90:0.089);
			\draw[dashed,-stealth] (-175:1) arc (-175:-75:1);
			\draw[dashed] (-76:1) arc (-76:-48.9396:1);
			\fill (48.9396:1) circle (1.5pt);
			\node[above right] at (48.9396:0.9) {$-\pi$};
			\fill (-48.9396:1) circle (1.5pt);
			\node[below right] at (-48.9396:1) {$\pi$};
			\fill (175:1) circle (1.5pt);
			\fill (-175:1) circle (1.5pt);
			\node[above left] at (177:0.9) {$-\theta_0$};
			\node[below left] at (-177:0.9) {$\theta_0$};
	\end{tikzpicture}}
	\caption{The path of $y^{E^\sigma}_1$ and $y^{E^\sigma}_2$ when $\theta\in[-\pi,\pi]$.}\label{pathofy}
\end{figure}
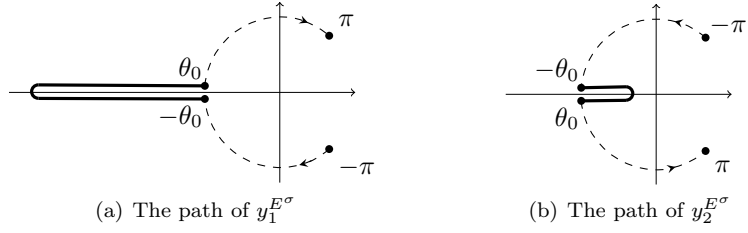
the paths of $y^{E^\sigma}_1$ and $y^{E^\sigma}_2$ when $\theta$ varies from $-\pi$ to $\pi$. (Take the principal branch of the square root.) In particular, we indicate with thick lines the paths when $\theta\in[-\theta_0,\theta_0].$ Let
\[\gamma_1^{E^\sigma}=\left\{\left(x,y_1^{E^\sigma}(x)\right)|\theta\in[-\theta_0,\theta_0]\right\},\ \ \ \gamma_2^{E^\sigma}=\left\{\left(x,y_2^{E^\sigma}(x)\right)|\theta\in[-\theta_0,\theta_0]\right\}.\]
Take $\gamma_{E^\sigma}=\gamma_1^{E^\sigma}\cup\gamma_2^{E^\sigma}$ with the orientation induced on $\gamma_1^{E^\sigma}$ by $\theta$ increasing and on $\gamma_2^{E^\sigma}$ by $\theta$ decreasing. One can check that $\gamma_{E^\sigma}$ is an element of $H_1(E^\sigma(\CC);\Z)^-$.

Let $M_1=\{x,y\}$ and $M_2=\{\phi^*(x),\phi^*(y)\}$, where the $x,y$ in $M_1$ are functions in $K(E)$ by \eqref{rationaltransformation} and the $x,y$ in $M_2$ are functions in $K(E^\sigma)$, they pull back by $\phi$ to functions in $K(E)$. Since our polynomial family $P_k(x,y)=x+\frac{1}{x}+y+\frac{1}{y}+k$ is tempered, i.e. the set of roots of all the face polynomials of $P_k$ consists of roots of unity only (see \cite{Vil97,LQ19} for more about this notion), we know from \cite[III.9]{Vil97} that $nM_1\in K_2(E;\Z)$ for some $n\in \N$. Therefore, one can use $M_1$ and $M_2$ to verify the second part of Beilinson's conjecture. The Beilinson regulator for $\gamma_E,\gamma_{E^\sigma}\in H_1(X;\Z)^-$ and $M_1,M_2$ now reads

\begin{align}\label{Beilinsonregulator}
	R&=\left|\det{\begin{pmatrix}\langle\gamma_E,M_1\rangle&\langle\gamma_{E^{\sigma}},M_1\rangle\\\langle\gamma_E,M_2\rangle&\langle\gamma_{E^\sigma},M_2\rangle\end{pmatrix}}\right|.
\end{align}

\begin{proof}[Proof of \eqref{thecasek=12pm8sqrt2}]
	By definition, we calculate that
	\begin{align*}
		\langle\gamma_E,M_1\rangle&=\frac{1}{2\pi}\int_{\gamma_E}\eta(x,y)\\
		&=\frac{1}{2\pi}\int_{\gamma_E}\log|x|\im\left(\frac{dy}{y}\right)-\log|y|\im\left(\frac{dx}{x}\right)\\
		&=-\frac{1}{2\pi}\int_{-\pi}^{\pi}\log|y^E_1(e^{ i\theta})|d\theta\\
		&=-m\left(12+8\sqrt{2}\right),
	\end{align*}
	where the last equality follows by Jensen's formula (see \cite[\S\S 1A]{Boyd98} or \cite[\S\S 3.2]{LR07}). Similarly, we have
	\[\langle\gamma_{E^{\sigma}},M_1\rangle=-2m\left(12-8\sqrt{2}\right).\]
	
	Also, for the second row of \eqref{Beilinsonregulator}, we have 
	\begin{align*}
		\langle\gamma_E,M_2\rangle&=\langle\gamma_E,\{\phi^*(x),\phi^*(y)\}\rangle\\
		&=\langle\phi_*\gamma_E,\{x,y\}\rangle
	\end{align*}
	and
	\begin{align*}
		\langle\gamma_{E^{\sigma}},M_2\rangle&=\langle\gamma_{E^{\sigma}},\{\phi^*(x)^\sigma,\phi^*(y)^\sigma\}\rangle\\
		&=\langle\gamma_{E^{\sigma}},\{(\phi^\sigma)^*(x),(\phi^\sigma)^*(y)\}\rangle\ \ \ (\text{here\ } x,y\in\CC(E(\CC))),\\
		&=\langle(\phi^\sigma)_*\gamma_{E^{\sigma}},\{x,y\}\rangle.
	\end{align*}
	In order to determine $\phi_*\gamma_E$ and $(\phi^\sigma)_*\gamma_{E^{\sigma}}$. One can calculate that
	
	\begin{equation}\label{diff1}
		\begin{split}
			\int_{\gamma_E}\omega_E&=\int_{\gamma_E}\frac{dX}{2Y}\\
			&=\int_{\gamma_E}\frac{ydx}{x(1-y^{2})}\ \ \ (\text{because\ } \frac{dX}{2Y}=\frac{ydx}{x(1-y^{2})}\ \text{by \eqref{inverserationaltransformation}})\\
			&=\int_{-\pi}^{\pi}\frac{y_1^E(e^{ i\theta})de^{ i\theta}}{e^{ i\theta}(1-y_1^E(e^{ i\theta})^2)}\\[5pt]
			&\approx0.27152i
		\end{split}
	\end{equation}
	and
	\begin{equation}\label{diff2}
		\begin{split}
			\int_{\gamma_{E^\sigma}}\omega_{E^\sigma}&=\int_{\gamma_1^{E^\sigma}}\omega_{E^\sigma}+\int_{\gamma_2^{E^\sigma}}\omega_{E^\sigma}\\[3pt]
			&=\int_{-\theta_0}^{\theta_0}\frac{y_1^{E^\sigma}(e^{ i\theta})de^{ i\theta}}{e^{ i\theta}(1-y_1^{E^\sigma}(e^{ i\theta})^2)}+\int_{\theta_0}^{-\theta_0}\frac{y_2^{E^\sigma}(e^{ i\theta})de^{ i\theta}}{e^{ i\theta}(1-y_2^{E^\sigma}(e^{ i\theta})^2)}\\[5pt]
			&\approx 3.1651 i.
		\end{split}
	\end{equation}
	Compare the above integrals of $\omega_E$ and $\omega_{E^\sigma}$ with the periods of lattices associated to $E$ and $E^\sigma$ calculated by \textsf{PARI/GP}, we can deduce that $\gamma_E$ and $\gamma_{E^\sigma}$ are in fact generators of $H_1(E;\Z)^-$ and $H_1(E^\sigma;\Z)^-$. Since $\phi$ and $\phi^\sigma$ are
	defined over $\R$, we know that $\phi_*$ maps $H_1(E(\CC);\Z)^-$ to $H_1(E^\sigma(\CC);\Z)^-$ and $(\phi^\sigma)_*$ maps $H_1(E^\sigma(\CC);\Z)^-$ to $H_1(E(\CC);\Z)^-$. Thus, there must exist $a,b\in\Z$ such that
	\[\phi_*\gamma_E=a\gamma_{E^\sigma},\ \ (\phi^\sigma)_*\gamma_{E^{\sigma}}=b\gamma_E.\]
	We have $ab=4$ because $\phi^\sigma\circ\phi=[4]$. According to \eqref{differential1} and \eqref{isomorphism}, we have
	\[\phi^*(\omega_{E^\sigma})=(\varphi\circ\psi)^*(\omega_{E^\sigma})=\psi^*\left(\frac{1}{u}\omega_{\widetilde E}\right)=\frac{1}{u}\omega_E.\]
	Therefore
	\begin{equation}\label{diff11}
		\int_{\phi_*\gamma_E}\omega_{E^\sigma}=\int_{\gamma_E}\phi^*(\omega_{E^\sigma})=\frac{1}{u}\int_{\gamma_E}\omega_E\approx 3.1651 i
	\end{equation}
	according to \eqref{diff1}. Since $a$ is an integer, by comparing \eqref{diff2} with \eqref{diff11}, it follows that $a=1$ and then $b=4$. 
	
	Immediately, we have
	\begin{align*}
		\langle\gamma_E,M_2\rangle&=\langle\gamma_{E^\sigma},M_1\rangle=-2m\left(12-8\sqrt{2}\right),\\
		\langle\gamma_{E^\sigma},M_2\rangle&=4\langle\gamma_{E},M_1\rangle=-4m\left(12+8\sqrt{2}\right),
	\end{align*}
	and
	\begin{equation}\label{RasMahlermeasure}
		R=4\left(m\left(12+8\sqrt{2}\right)^2-m\left(12-8\sqrt{2}\right)^2\right).
	\end{equation}
	Beilinson's conjecture implies that $R$ should be some rational multiple of $\pi^{-4}L(E,2).$
	
	In \cite{BF18}, Bruin and Ferraguti described an algorithm to express explicitly the $L$-functions of elliptic curves completely defined over quadratic number fields without CM as the product of two $L$-functions associated to a pair of conjugate newforms. Fortunately, in our CM case, we can seek help from the LMFDB database \cite{database}.
	
	Since $E$ has $j$-invariant $41113158120 + 29071392966\sqrt{2}$ and conductor norm $32$. One can search in LMFDB and find that $E$ is isomorphic over $K$ to
	the elliptic curve with LMFDB label \href{https://www.lmfdb.org/EllipticCurve/2.2.8.1/32.1/a/8}{\texttt{2\!.\!2\!.\!8\!.\!1-32\!.\!1-a8}}. Also, its $L$-function with label \href{https://www.lmfdb.org/L/4/2e11/1.1/c1e2/0/0}{\texttt{4-2e11-1\!.\!1-c1e2-0-0}} can be written as
	\[L(E,s)=L(f_{64},s)L(f_{32},s),\]
	where
	\begin{align*}
		f_{64}&=\frac{\eta(8\tau)^8}{\eta(4\tau)^2\eta(16\tau)^2}=q+2q^5-3q^9-6q^{13}+2q^{17}-q^{25}+\cdots\in\mathcal S_2(\Gamma_0(64)),\\[5pt]
		f_{32}&=\eta(4\tau)^2\eta(8\tau)^2=q-2q^5-3q^9+6q^{13}+2q^{17}-q^{25}+\cdots\in\mathcal S_2(\Gamma_0(32)).
	\end{align*}
	One can use Sturm bound to prove that
	\begin{align*}
		&\frac{1}{2}\underset{m,n\in\Z}{{\sum}}\chi_{-4}(n)nq^{16m^2+n^2}=\frac{f_{64}+f_{32}}{2},\\
		\frac{1}{4}\underset{m,n\in\Z}{{\sum}}&\chi_{-4}(n)(8m+5n)q^{16m^2+16mn+5n^2}=\frac{f_{64}-f_{32}}{4}.
	\end{align*}
	According to \eqref{RasMahlermeasure} as well as to Table \ref{mahlermeasureandmodularforms}, we finally have
	\begin{align*}
		R&=4\left(\left(\frac{32}{\pi^2}\cdot\frac{L(f_{64},2)+L(f_{32},2)}{2}\right)^2\!-\left(\frac{64}{\pi^2}\cdot\frac{L(f_{64},2)-L(f_{32},2)}{4}\right)^2\right)\\[2pt]
		&=\frac{4096}{\pi^4}L(f_{64},2)L(f_{32},2)\\[2pt]
		&=\frac{4096}{\pi^4}L(E,2).
	\end{align*}
	Now, by a simplification, we can obtain \eqref{thecasek=12pm8sqrt2}.
\end{proof}

\section{The other four cases}\label{other4cases}

For the remaining four cases $k=\sqrt{2}\pm\sqrt{6},\ 4\sqrt{2}\pm4\sqrt{6},\ \frac{3\sqrt{2}}{2}\pm\frac{\sqrt{14}}{2},\ \text{and}\ 24\sqrt{2}\pm8\sqrt{14}$, there will be some trouble if we continue to use the rational transformation \eqref{rationaltransformation}. Since for instance when $k=\sqrt{2}+\sqrt{6},\ E_k$ has Weierstrass equation $Y^2=X^3+\sqrt{3}X^2+X$ under \eqref{rationaltransformation}. It is defined over $K=\Q(\sqrt{3})$. However, under this transformation, the functions $x=\frac{(\sqrt{2}+\sqrt{6})X-2Y}{2X(X-1)}$ and $y=\frac{(\sqrt{2}+\sqrt{6})X+2Y}{2X(X-1)}$ are not defined over $K$. It turns out that $\{x,y\}$ should not be an element of $K_2^T(E_k)$ by definition.

To settle this issue, we turn to 
the rational transformation
\begin{equation}\label{rationaltransformation2}
	x=\frac{kX-\sqrt{2}Y}{X(X-2)},\ \ \ \ \ y=\frac{kX+\sqrt{2}Y}{X(X-2)}.
\end{equation}
Under this transformation, we obtain elliptic curves
\begin{equation*}
	E'_k:Y^2=X^3+\left(\frac{k^2}{2}-4\right)X^2+4X.
\end{equation*}
And one can easily check that the functions $xy$ and $\frac{x}{y}$ are now defined over $K$ for our four interested pairs of $k$. So we will use
\begin{equation}\label{K2elements}
	M_1=\left\{xy,\frac{x}{y}\right\}\ \ \text{and}\ \  M_2=\left\{\phi^*(xy),\phi^*\left(\frac{x}{y}\right)\right\}.
\end{equation}
to compute the Beilinson regulators in this section.

\begin{proof}[Proof of \eqref{thecasek=sqrt2pmsqrt6}]
	Under \eqref{rationaltransformation2}, we have
	\begin{align*}
		E'_{\sqrt{2}+\sqrt{6}}&:Y^2=X^3+2\sqrt{3}X^2+4X,\\
		E'_{\sqrt{2}-\sqrt{6}}&:Y^2=X^3-2\sqrt{3}X^2+4X.
	\end{align*}
	They are defined over $K=\Q(\sqrt{3})$. Let $\sigma\in\Gal(K/\Q)$ be the nontrivial element. Then we can write $E$ for $E'_{\sqrt{2}+\sqrt{6}}$ and $E^{\sigma}$ for $E'_{\sqrt{2}-\sqrt{6}}$. There is an isogeny $\phi:E\to E^\sigma$  defined over $K$ of degree $3$, it is given by
	\[(X,Y)\mapsto\left(\frac{3X \left(X^2+4 \sqrt{3} X+12\right)}{\left(3 X+2 \sqrt{3}\right)^2},\frac{3 \sqrt{3}Y\left(X+2 \sqrt{3}\right) \left(X^2+4\right)}{\left(3 X+2\sqrt{3}\right)^3}\right).\]
	We have $\phi^\sigma\circ\phi=[-3]$ and $\phi^*(\omega_{E^\sigma})=\sqrt{3} \hspace{1pt}\omega_E$.
	
	Take
	\[\gamma_E=\left\{\left(x,y_1^E\right)|\theta\in[-\theta_1,\theta_1]\right\}\cup\left\{\left(x,y_2^E\right)|\theta\in[-\theta_1,\theta_1]\right\}\in H_1(E;\Z)^-\]
	and
	\[\gamma_{E^\sigma}\!=\!\left\{\left(x,y_1^{E^\sigma}\right)|\theta\in[\theta_2,2\pi-\theta_2]\right\}\!\cup\!\left\{\left(x,y_2^{E^\sigma}\right)|\theta\in[\theta_2,2\pi-\theta_2]\right\}\!\in\! H_1(E^\sigma;\Z)^-,\]
	where
	\[\theta_1=\pi-\arctan\sqrt{\frac{\left(\sqrt{2}-1\right)\left(\sqrt{3}-1\right)}{2}},\ \  \theta_2=\pi-\arctan\sqrt{\frac{\left(\sqrt{2}+1\right)\left(\sqrt{3}+1\right)}{2}}.\]
	By calculating the integral of invariant differential, we know that $\gamma_E$ and $\gamma_{E^\sigma}$ are generators of $H_1(X;\Z)^-$ and 
	\begin{equation}\label{path}
		\phi_*\gamma_E=\pm3\gamma_{E^\sigma},\ \ \ (\phi^\sigma)_*\gamma_{E^{\sigma}}=\mp\gamma_E.
	\end{equation}

	Next, take $M_1$ and $M_2$ as \eqref{K2elements}. In this setting, we have
	\begin{align*}
		\langle\gamma_E,M_1\rangle&=\frac{1}{2\pi}\int_{\gamma_E}\eta\left(xy,\frac{x}{y}\right)\nonumber\\
		&=\frac{1}{2\pi}\int_{\gamma_E}\log|xy|\im\left(\frac{d\frac{x}{y}}{\frac{x}{y}}\right)-\log\left|\frac{x}{y}\right|\im\left(\frac{dxy}{xy}\right)\\
		&=\frac{1}{2\pi}\int_{\gamma_E}\log|y|\im\left(-\frac{dy}{y}+\frac{dx}{x}\right)+\log\left|y\right|\im\left(\frac{xdy+ydx}{xy}\right)\\
		&=\frac{1}{\pi}\int_{\gamma_E}\log|y|\im\left(\frac{dx}{x}\right)\\
		&=\pm4m\left(\!\sqrt{2}+\sqrt{6}\right).
	\end{align*}
	Similarly, we have $\langle\gamma_{E^\sigma},M_1\rangle=\pm4m\left(\!\sqrt{2}-\sqrt{6}\hspace{1pt}\right)$. Also, according to \eqref{path} and the calculations in the previous section, we have
	\[\langle\gamma_E,M_2\rangle=\pm12m\left(\!\sqrt{2}-\sqrt{6}\right),\ \ \ \ \langle\gamma_{E^\sigma},M_2\rangle=\mp4m\left(\!\sqrt{2}+\sqrt{6}\right).\]

	We search in \cite{database} and find that $E$ is isomorphic over $K$ to
	the elliptic curve \href{https://www.lmfdb.org/EllipticCurve/2.2.12.1/16.1/a/1}{\texttt{2\!.\!2\!.\!12\!.\!1-16\!.\!1-a1}}, its $L$-function
	\href{https://www.lmfdb.org/L/4/48e2/1.1/c1e2/0/2}{\texttt{4-48e2-1\!.\!1-c1e2-0-2}} can be written as
	\[L(E,s)=L(f_{48},s)L(g_{48},s),\]
	where $f_{48}$ is the newform \href{https://www.lmfdb.org/ModularForm/GL2/Q/holomorphic/48/2/c/a/47/1/}{\texttt{48\!.\!2\!.\!c\!.\!a\!.\!47\!.\!1}} and $g_{48}$ is its dual form \href{https://www.lmfdb.org/ModularForm/GL2/Q/holomorphic/48/2/c/a/47/2/}{\texttt{48\!.\!2\!.\!c\!.\!a\!.\!47\!.\!2}}. One can use Sturm bound to prove that
	\begin{align*}
		-\frac{1}{6}\underset{m,n\in\Z}{{\sum}}\chi_{-4}(n)(m-2n)q^{m^2-mn+n^2}&=\frac{3+ i\sqrt{3}}{6}f_{48}+\frac{3- i\sqrt{3}}{6}g_{48},\\
		-\frac{1}{2}\underset{m,n\in\Z}{{\sum}}\chi_{-4}(n)(3m-2n)q^{3m^2-3mn+n^2}&=\frac{1- i\sqrt{3}}{2}f_{48}+\frac{1+ i\sqrt{3}}{2}g_{48}.
	\end{align*}
	Hence, we have
	\begin{align*}
		R&=\left|\det{\begin{pmatrix}\langle\gamma_E,M_1\rangle&\langle\gamma_{E^{\sigma}},M_1\rangle\\\langle\gamma_E,M_2\rangle&\langle\gamma_{E^\sigma},M_2\rangle\end{pmatrix}}\right|\\
		&=16m\left(\!\sqrt{2}+\sqrt{6}\right)^2+48m\left(\!\sqrt{2}-\sqrt{6}\right)^2\\
		&=16\left(\frac{6\sqrt{3}}{\pi^2}\cdot\left(\frac{3+ i\sqrt{3}}{6}L(f_{48},2)+\frac{3- i\sqrt{3}}{6}L(g_{48},2)\right)\right)^2\\
		&+48\left(\frac{2\sqrt{3}}{\pi^2}\cdot\left(\frac{1- i\sqrt{3}}{2}L(f_{48},2)+\frac{1+ i\sqrt{3}}{2}L(g_{48},2)\right)\right)^2\\
		&=\frac{2304}{\pi^4}L(f_{48},2)L(g_{48},2)\\[2pt]
		&=\frac{2304}{\pi^4}L(E,2).
	\end{align*}
	A simplification immediately yields \eqref{thecasek=sqrt2pmsqrt6}. 
\end{proof}

\begin{proof}[Proof of \eqref{thecasek=4sqrt2pm4sqrt6}]
	Under \eqref{rationaltransformation2}, we have
	\begin{align*}
		E'_{4\sqrt{2}+4\sqrt{6}}&:Y^2=X^3+(60+32\sqrt{3})X^2+4X,\\
		E'_{4\sqrt{2}-4\sqrt{6}}&:Y^2=X^3+(60-32\sqrt{3})X^2+4X.
	\end{align*}
	They are defined over $K=\Q(\sqrt{3})$. Let $\sigma\in\Gal(K/\Q)$ be the nontrivial element. Then we can write $E$ for $E'_{4\sqrt{2}+4\sqrt{6}}$ and $E^{\sigma}$ for $E'_{4\sqrt{2}-4\sqrt{6}}$. There is an isogeny $\phi:E\to E^\sigma$  defined over $K$ of degree $3$, it is given by $(X,Y)\mapsto\left(\phi_1(X),Y\phi_2(X)\right),$ where
	\begin{align*}
		\phi_1(X)&=\frac{3X\left(\left(7-4\sqrt{3}\right)X^2+\left(12-8\sqrt{3}\right)X+12\right)}{\left(3X+6-4\sqrt{3}\right)^2},\\
		\phi_2(X)&=\frac{\left(2 \sqrt{3}-3\right)^3\left(X-6-4 \sqrt{3}\right) \left(X^2+12 X+4\right)}{\left(3 X+6-4 \sqrt{3}\right)^3}.
	\end{align*}
	We have $\phi^\sigma\circ\phi=[-3]$ and $\phi^*(\omega_{E^\sigma})=(3+2\sqrt{3})\omega_E$.
	
	Take
	\[\gamma_E=\left\{\left(x,y_1^E\right)|\theta\in[-\pi,\pi]\right\}\in H_1(E;\Z)^-\]
	and
	\[\gamma_{E^\sigma}=\left\{\left(x,y_1^{E^\sigma}\right)|\theta\in[-\pi,\pi]\right\}\in H_1(E^\sigma;\Z)^-.\]
	By calculating the integral of invariant differential, we know that $\gamma_E$ and $\gamma_{E^\sigma}$ are generators of $H_1(X;\Z)^-$ and
	\[\phi_*\gamma_E=\mp\gamma_{E^\sigma},\ \ \ (\phi^\sigma)_*\gamma_{E^{\sigma}}=\pm3\gamma_E.\]
	
	Next, take $M_1$ and $M_2$ as \eqref{K2elements}. We have
	\begin{align*}
		\langle\gamma_E,M_1\rangle&=\pm2m\left(\!4\sqrt{2}+4\sqrt{6}\right),\ \ \ \ \langle\gamma_{E^\sigma},M_1\rangle=\pm2m\left(\!4\sqrt{2}-4\sqrt{6}\right),\\
		\langle\gamma_E,M_2\rangle&=\mp2m\left(\!4\sqrt{2}-4\sqrt{6}\right),\ \ \ \ \langle\gamma_{E^\sigma},M_2\rangle=\pm6m\left(\!4\sqrt{2}+4\sqrt{6}\right).
	\end{align*}
	
	We search in \cite{database} and find that $E$ is isomorphic over $K$ to
	the elliptic curve \href{https://www.lmfdb.org/EllipticCurve/2.2.12.1/256.1/c/8}{\texttt{2\!.\!2\!.\!12\!.\!1-256\!.\!1-c8}}, its $L$-function
	\href{https://www.lmfdb.org/L/4/192e2/1.1/c1e2/0/1}{\texttt{4-192e2-1\!.\!1-c1e2-0-1}} can be written as
	\[L(E,s)=L(f_{192},s)L(g_{192},s),\]
	where $f_{192}$ is the newform \href{https://www.lmfdb.org/ModularForm/GL2/Q/holomorphic/192/2/c/a/191/1/}{\texttt{192\!.\!2\!.\!c\!.\!a\!.\!191\!.\!1}} and $g_{192}$ is its dual form \href{https://www.lmfdb.org/ModularForm/GL2/Q/holomorphic/192/2/c/a/191/2/}{\texttt{192\!.\!2\!.\!c\!.\!a\!.\!191\!.\!2}}. One can use Sturm bound to prove that
	\begin{align*}
		\frac{1}{2}\underset{m,n\in\Z}{{\sum}}\chi_{-4}(n)nq^{12m^2+n^2}&=\frac{f_{192}+g_{192}}{2},\\
		\frac{1}{2}\underset{m,n\in\Z}{{\sum}}\chi_{-4}(n)nq^{4m^2+3n^2}&=\frac{g_{192}-f_{192}}{2 i\sqrt{3}}.
	\end{align*}
	Hence, we have
	\begin{align*}
		R&=12m\left(\!4\sqrt{2}+4\sqrt{6}\right)^2+4m\left(\!4\sqrt{2}+4\sqrt{6}\right)^2\\
		&=12\left(\!\frac{16\sqrt{3}}{\pi^2}\cdot\frac{L(f_{192},2)+L(g_{192},2)}{2}\right)^2\!+4\left(\!\frac{48\sqrt{3}}{\pi^2}\cdot\frac{L(g_{192},2)-L(f_{192},2)}{2 i\sqrt{3}}\right)^2\\
		&=\frac{9216}{\pi^4}L(f_{192},2)L(g_{192},2)\\[2pt]
		&=\frac{9216}{\pi^4}L(E,2).
	\end{align*}
	A simplification immediately yields \eqref{thecasek=4sqrt2pm4sqrt6}.
\end{proof}


\begin{proof}[Proof of \eqref{thecasek=3sqrt2/2pmsqrt14/2}]
	Under \eqref{rationaltransformation2}, we have
	\begin{align*}
		E'_{\frac{3\sqrt{2}}{2}+\frac{\sqrt{14}}{2}}&:Y^2=X^3+\frac{3\sqrt{7}}{2}X^2+4X,\\
		E'_{\frac{3\sqrt{2}}{2}-\frac{\sqrt{14}}{2}}&:Y^2=X^3-\frac{3\sqrt{7}}{2}X^2+4X.
	\end{align*}
	They are defined over $K=\Q(\sqrt{7})$. Let $\sigma\in\Gal(K/\Q)$ be the nontrivial element. Then we can write $E$ for $E'_{\frac{3\sqrt{2}}{2}+\frac{\sqrt{14}}{2}}$ and $E^{\sigma}$ for $E'_{\frac{3\sqrt{2}}{2}-\frac{\sqrt{14}}{2}}$. There is an isogeny $\phi:E\to E^\sigma$  defined over $K$ of degree $7$, it is given by $(X,Y)\mapsto\left(\frac{7X\phi_2(X)}{\phi_1(X)^2},\frac{7\sqrt{7}Y\phi_3(X)}{\phi_1(X)^3}\right),$ where
	\begin{align*}
		\phi_1(X)&=7X^3+14\sqrt{7}X^2+56X+8\sqrt{7},\\
		\phi_2(X)&=X^6+8 \sqrt{7} X^5+168 X^4+240 \sqrt{7} X^3+1232 X^2+448 \sqrt{7}X+448,\\
		\phi_3(X)&=X^9+6 \sqrt{7} X^8+96 X^7+120 \sqrt{7} X^6+720 X^5+544 \sqrt{7} X^4+2752 X^3\\
		&+1536 \sqrt{7} X^2+3584 X+512 \sqrt{7}.
	\end{align*}
	We have $\phi^\sigma\circ\phi=[-7]$ and $\phi^*(\omega_{E^\sigma})=\sqrt{7} \hspace{1pt}\omega_E$.
	
	Take
	\[	\gamma_E=\left\{\left(x,y_1^E\right)|\theta\in[-\theta_1,\theta_1]\right\}\cup\left\{\left(x,y_2^E\right)|\theta\in[-\theta_1,\theta_1]\right\}\in H_1(E;\Z)^-\]
	and
	\[\gamma_{E^\sigma}=\left\{\left(x,y_1^{E^\sigma}\right)|\theta\in[-\theta_2,\theta_2]\right\}\cup\left\{\left(x,y_2^{E^\sigma}\right)|\theta\in[-\theta_2,\theta_2]\right\}\in H_1(E^\sigma;\Z)^-,\]
	where
	\begin{align*}
		\theta_1&=\pi-\arctan\frac{\sqrt{552\sqrt{2}-433-4\sqrt{7 \left(2993-1428 \sqrt{2}\hspace{1pt}\right)}}}{47},\\[5pt]
		\theta_2&=\arctan\frac{\sqrt{552\sqrt{2}-433+4\sqrt{7 \left(2993-1428 \sqrt{2}\hspace{1pt}\right)}}}{47}.
	\end{align*}
	By calculating the integral of invariant differential, we know that $\gamma_E$ and $\gamma_{E^\sigma}$ are generators of $H_1(X;\Z)^-$ and
	\[\phi_*\gamma_E=\pm7\gamma_{E^\sigma},\ \ \ (\phi^\sigma)_*\gamma_{E^{\sigma}}=\mp\gamma_E.\]
	
	Next, take $M_1$ and $M_2$ as \eqref{K2elements}. We have
	\begin{align*}
		\langle\gamma_E,M_1\rangle&=\pm4m\left(\!\frac{3\sqrt{2}}{2}+\frac{\sqrt{14}}{2}\right),\ \ \ \ \langle\gamma_{E^\sigma},M_1\rangle=\pm4m\left(\!\frac{3\sqrt{2}}{2}-\frac{\sqrt{14}}{2}\right),\\
		\langle\gamma_E,M_2\rangle&=\pm28m\left(\!\frac{3\sqrt{2}}{2}-\frac{\sqrt{14}}{2}\right),\ \ \ \ \langle\gamma_{E^\sigma},M_2\rangle=\mp4m\left(\!\frac{3\sqrt{2}}{2}+\frac{\sqrt{14}}{2}\right).
	\end{align*}
	
	We search in \cite{database} and find that $E$ is isomorphic over $K$ to
	the elliptic curve \href{https://www.lmfdb.org/EllipticCurve/2.2.28.1/1.1/a/2}{\texttt{2\!.\!2\!.\!28\!.\!1-1\!.\!1-a2}}, its $L$-function
	\href{https://www.lmfdb.org/L/4/28e2/1.1/c1e2/0/1}{\texttt{4-28e2-1\!.\!1-c1e2-0-1}} can be written as
	\[L(E,s)=L(f_{28},s)L(g_{28},s),\]
	where $f_{28}$ is the newform \href{https://www.lmfdb.org/ModularForm/GL2/Q/holomorphic/28/2/d/a/27/1/}{\texttt{28\!.\!2\!.\!d\!.\!a\!.\!27\!.\!1}} and $g_{28}$ is its dual form \href{https://www.lmfdb.org/ModularForm/GL2/Q/holomorphic/28/2/d/a/27/2/}{\texttt{28\!.\!2\!.\!d\!.\!a\!.\!27\!.\!2}}.
	One can use Sturm bound to prove that
	{\small
		\begin{align*}
			-\frac{1}{14}\underset{m,n\in\Z}{{\sum}}\chi_{-4}(n)(m\!-\!4n)q^{m^2-mn+2n^2}&\!=\alpha_1 f_{28}(2\tau)\!+\!\bar{\alpha}_1g_{28}(2\tau)\!+\!\beta_1 f_{28}(4\tau)\!+\!\bar{\beta}_1g_{28}(4\tau),\\
			-\frac{1}{2}\underset{m,n\in\Z}{{\sum}}\chi_{-4}(n)(7m\!-\!4n)q^{7m^2-7mn+2n^2}&\!=\alpha_2 f_{28}(2\tau)\!+\!\bar{\alpha}_2g_{28}(2\tau)\!+\!\beta_2 f_{28}(4\tau)\!+\!\bar{\beta}_2g_{28}(4\tau),
		\end{align*}
	}
	where $\alpha_1=\frac{7- i\sqrt{7}}{14},\beta_1=\frac{7+3 i\sqrt{7}}{7},\alpha_2=\frac{1+ i\sqrt{7}}{2},\beta_2=-3- i\sqrt{7}.$ Hence, we have
	{\small
		\begin{align*}
			R&=16m\left(\frac{3\sqrt{2}}{2}+\frac{\sqrt{14}}{2}\right)^2+112m\left(\frac{3\sqrt{2}}{2}-\frac{\sqrt{14}}{2}\right)^2\\
			&=\!16\!\left(\!\frac{14\sqrt{7}}{\pi^2}\!\left(\alpha_1L(f_{28}(2\tau),2)\!+\!\bar\alpha_1L(g_{28}(2\tau),2)\!+\!\beta_1L(f_{28}(4\tau),2)\!+\!\bar\beta_1L(g_{28}(4\tau),2)\right)\!\right)^2\\
			&+\!112\!\left(\!\frac{2\sqrt{7}}{\pi^2}\!\left(\alpha_2L(f_{28}(2\tau),2)\!+\!\bar\alpha_2L(g_{28}(2\tau),2)\!+\!\beta_2L(f_{28}(4\tau),2)\!+\!\bar\beta_2L(g_{28}(4\tau),2)\right)\!\right)^2\\
			&=\frac{21952}{\pi^4}\left(\left(\frac{\alpha_1}{2^2}+\frac{\beta_1}{4^2}\right)L(f_{28},2)+\left(\frac{\bar\alpha_1}{2^2}+\frac{\bar\beta_1}{4^2}\right)L(g_{28},2)\right)^2\\
			&+\frac{3136}{\pi^4}\left(\left(\frac{\alpha_2}{2^2}+\frac{\beta_2}{4^2}\right)L(f_{28},2)+\left(\frac{\bar\alpha_2}{2^2}+\frac{\bar\beta_2}{4^2}\right)L(g_{28},2)\right)^2\\
			&=\frac{3136}{\pi^4}L(f_{28},2)L(g_{28},2)\\
			&=\frac{3136}{\pi^4}L(E,2).
		\end{align*}
	}
	A simplification immediately yields \eqref{thecasek=3sqrt2/2pmsqrt14/2}.
\end{proof}


\begin{proof}[Proof of \eqref{thecasek=24sqrt2pm8sqrt14}]
	Under \eqref{rationaltransformation2}, we have
	\begin{align*}
		E'_{24\sqrt{2}+8\sqrt{14}}&:Y^2=X^3+(1020+384\sqrt{7})X^2+4X,\\
		E'_{24\sqrt{2}-8\sqrt{14}}&:Y^2=X^3+(1020-384\sqrt{7})X^2+4X.
	\end{align*}
	They are defined over $K=\Q(\sqrt{7})$. Let $\sigma\in\Gal(K/\Q)$ be the nontrivial element. Then we can write $E$ for $E'_{24\sqrt{2}+8\sqrt{14}}$ and $E^{\sigma}$ for $E'_{24\sqrt{2}-8\sqrt{14}}$. There is an isogeny $\phi:E\to E^\sigma$  defined over $K$ of degree $7$, it is given by $(X,Y)\mapsto\left(\frac{7X\phi_2(X)}{\phi_1(X)^2},\frac{-(21-8\sqrt{7})^3Y\phi_3(X)}{\phi_1(X)^3}\right),$ where
	\begin{align*}
		\phi_1(X)&=7X^3-14\left(15+8\sqrt{7}\right)X^2+308X+8\left(21-8\sqrt{7}\right),\\
		\phi_2(X)&=\left(127-48\sqrt{7}\right)X^6+44\left(21-8\sqrt{7}\right)X^5+12\left(371-32\sqrt{7}\right)X^4\\
		&-\!96\!\left(189\!+\!104\sqrt{7}\right)\!X^3\!+\!560\!\left(139\!+\!48\sqrt{7}\right)\!X^2\!-\!448\!\left(15\!+\!8\sqrt{7}\right)\!X\!+\!448,\\
		\phi_3(X)&\!=\!X^9\!-\!6\!\left(15\!+\!8\sqrt{7}\right)\!X^8\!-\!48\!\left(6289\!+\!2380\sqrt{7}\right)\!X^7\!-\!96\!\left(212583\!+\!80356 \sqrt{7}\right)\!X^6\\
		&+288\left(2959+1112 \sqrt{7}\right) X^5+64\left(13721325+5186128 \sqrt{7}\right) X^4\\
		&+256 \left(2104321+795348 \sqrt{7}\right) X^3+1536 \left(1029+388 \sqrt{7}\right) X^2\\
		&-256\left(6727+2544 \sqrt{7}\right) X-512\left(21+8\sqrt{7}\right).
	\end{align*}
	We have $\phi^\sigma\circ\phi=[-7]$ and $\phi^*(\omega_{E^\sigma})=(21+8 \sqrt{7})\omega_E$.
	
	Take
	\[\gamma_E=\left\{\left(x,y_1^E\right)|\theta\in[-\pi,\pi]\right\}\in H_1(E;\Z)^-\]
	and
	\[\gamma_{E^\sigma}=\left\{\left(x,y_1^{E^\sigma}\right)|\theta\in[-\pi,\pi]\right\}\in H_1(E^\sigma;\Z)^-.\]
	By calculating the integral of invariant differential, we know that $\gamma_E$ and $\gamma_{E^\sigma}$ are generators of $H_1(X;\Z)^-$ and
	\[\phi_*\gamma_E=\pm\gamma_{E^\sigma},\ \ \ (\phi^\sigma)_*\gamma_{E^{\sigma}}=\mp7\gamma_E.\]
	
	Next, take $M_1$ and $M_2$ as \eqref{K2elements}. We have
	\begin{align*}
		\langle\gamma_E,M_1\rangle&=\pm2m\left(\!24\sqrt{2}+8\sqrt{14}\right),\ \ \ \ \langle\gamma_{E^\sigma},M_1\rangle=\pm2m\left(\!24\sqrt{2}-8\sqrt{14}\right),\\
		\langle\gamma_E,M_2\rangle&=\pm2m\left(\!24\sqrt{2}-8\sqrt{14}\right),\ \ \ \ \langle\gamma_{E^\sigma},M_2\rangle=\mp14m\left(\!24\sqrt{2}+8\sqrt{14}\right),
	\end{align*}
	
	We search in \cite{database} and find that $E$ is isomorphic over $K$ to
	the elliptic curve \href{https://www.lmfdb.org/EllipticCurve/2.2.28.1/256.1/j/8}{\texttt{2\!.\!2\!.\!28\!.\!1-256\!.\!1-j8}}, its $L$-function
	\href{https://www.lmfdb.org/L/4/448e2/1.1/c1e2/0/0}{\texttt{4-448e2-1\!.\!1-c1e2-0-0}} can be written as
	\[L(E,s)=L(f_{448},s)L(g_{448},s),\]
	where $f_{448}$ is the newform \href{https://www.lmfdb.org/ModularForm/GL2/Q/holomorphic/448/2/f/b/447/1/}{\texttt{448\!.\!2\!.\!f\!.\!b\!.\!447\!.\!1}} and $g_{448}$ is its dual form \href{https://www.lmfdb.org/ModularForm/GL2/Q/holomorphic/448/2/f/b/447/2/}{\texttt{448\!.\!2\!.\!f\!.\!b\!.\!447\!.\!2}}. One can use Sturm bound to prove that
	\begin{align*}
		\frac{1}{2}\underset{m,n\in\Z}{{\sum}}\chi_{-4}(n)nq^{28m^2+n^2}&=\frac{f_{448}+g_{448}}{2},\\
		\frac{1}{2}\underset{m,n\in\Z}{{\sum}}\chi_{-4}(n)nq^{4m^2+7n^2}&=\frac{g_{448}-f_{448}}{2 i\sqrt{7}}.
	\end{align*}
	Hence, we have 
	\begin{align*}
		R&=28m\left(24\sqrt{2}+8\sqrt{14}\right)^2+4m\left(24\sqrt{2}-8\sqrt{14}\right)^2\\
		&=28\!\left(\!\frac{16\sqrt{7}}{\pi^2}\cdot\frac{L(f_{448},2)+L(g_{448},2)}{2}\!\right)^2\!+4\!\left(\!\frac{112\sqrt{7}}{\pi^2}\cdot\frac{L(g_{448},2)-L(f_{448},2)}{2 i\sqrt{7}}\!\right)^2\\
		&=\frac{50176}{\pi^4}L(f_{448},2)L(g_{448},2)\\[2pt]
		&=\frac{50176}{\pi^4}L(E,2).
	\end{align*}
	A simplification immediately yields \eqref{thecasek=24sqrt2pm8sqrt14}.
\end{proof}


\begin{thebibliography}{99}

	\bibitem{BG86} S. Bloch, D. Grayson, {\it $K_2$ and $L$-functions of elliptic curves: computer calculations}, in \textit{Applications of Algebraic K-Theory to Algebraic Geometry and Number Theory}, Contemporary Mathematics, vol. 55, pp. 79--88, American Mathematical Society, Providence, RI, 1986.
	
	\bibitem{Boyd98} D. W. Boyd, {\it Mahler's measure and special values of $L$-functions},  Experiment. Math. {\bf 7} (1998), no. 1, 37--82.
	
	\bibitem{BF18} P. Bruin, A. Ferraguti, {\it On $L$-functions of quadratic $\Q$-curves}, Math. Comp. {\bf 87} (2018), no. 309, 459--499.
	
	\bibitem{Bru16} F. Brunault, {\it Regulators of Siegel units and applications}, J. Number Theory {\bf 163} (2016), 542--569.
	
	\bibitem{BZ20} F. Brunault, W. Zudilin, {\it Many variations of Mahler measures. A lasting symphony}, Australian Mathematical Society Lecture Series, vol. 28. Cambridge University Press, Cambridge, 2020. xv+167 pp.
	
	\bibitem{Mfaca} H. Cohen, F. Strömberg, {\it Modular forms. A classical approach}, Graduate Studies in Mathematics 179, American Mathematical Society, Providence, RI, 2017, xii+700 pp.
	
	\bibitem{Cox} D. A. Cox, {\it Primes of the form $x^2+ny^2$. Fermat, class field theory, and complex multiplication,} 2nd ed., John Wiley \& Sons, Inc., Hoboken, NJ, 2013, xviii+356 pp.
		
	\bibitem{Den97} C. Deninger, {\it Deligne periods of mixed motives, K-theory and the entropy of certain $\Z^n$-actions}, J. Amer. Math. Soc. {\bf 10} (1997), no. 2, 259--281.
	
	\bibitem{DJZ06} T. Dokchitser, R. de Jeu, and D. Zagier {\it Numerical verification of Beilinson's conjecture for $K_2$ of hyperelliptic curves}, Compositio Math. {\bf 142} (2006), no. 2, 339--373.
	
	\bibitem{Gal12} S. D. Galbraith, {\it Mathematics of public key cryptography}, Cambridge University Press, Cambridge, 2012. xiv+615 pp.
	
	\bibitem{Guo} X. Guo, Q. Ji, H. Liu, and H. Qin, {\it The Mahler measure of $x+1/x+y+1/y+4\pm4 \sqrt{2}$
	and Beilinson's conjecture}, Int. J. Number Theory, published online (2023).
	
	\bibitem{Heegner} K. Heegner, {\it Diophantische Analysis und Modulfunktionen}, (German) Math. Z. {\bf 56} (1952), 227--253.
	
	\bibitem{Piseries} T. Huber, D. Schultz, and D. Ye, {\it Series for $1/\pi$ of level $20$}, J. Number Theory {\bf 188} (2018), 121--136.
	
	\bibitem{Lal10} M. N. Lalín, {\it On a conjecture by Boyd}, Int. J. Number Theory {\bf 6} (2010), no. 3, 705--711.
	
	\bibitem{LR07} M. N. Lalín, M. D. Rogers, {\it Functional equations for Mahler measures of genus-one curves}, Algebra Number Theory {\bf 1}(1) (2007), 87--117.
	
	\bibitem{LSZ16} M. N. Lalín, D. Samart, and W. Zudilin, {\it Further explorations of Boyd's conjectures and a conductor $21$ elliptic curve}, J. Lond. Math. Soc. {\bf 93} (2016), no. 2, 341--360.
	
	\bibitem{LJ15} H. Liu, R. de Jeu, {\it On $K_2$ of certain families of curves}, Int. Math. Res. Notices (2015), no. 21, 10929--10958.
	
	\bibitem{LQ19} H. Liu, H. Qin, {\it Mahler measure of families of polynomials defining genus $2$ and $3$ curves}, Exp. Math. (2021), DOI: 10.1080/10586458.2021.1926014.
	
	\bibitem{RZ14} M. Rogers, W. Zudilin, {\it On the Mahler measure of $1+ X + 1/X + Y + 1/Y$}, Int. Math. Res. Notices (2014), no. 9, 2305--2326.
	
	\bibitem{Sam21} D. Samart, {\it A functional identity for Mahler measures of non-tempered polynomials}, Integral Transforms Spec. Funct. {\bf 32} (2021), no. 1, 78--89.
	
	\bibitem{Sam15} D. Samart, {\it Mahler measures as linear combinations of $L$-values of multiple modular forms}, Canad. J. Math. {\bf 67} (2015), no. 2, 424--449.
	
	\bibitem{Stark} H. M. Stark, {\it A complete determination of the complex quadratic fields of class-number one}, Michigan Math. J. {\bf 14} (1967), 1--27.
	
	\bibitem{Stark2} H. M. Stark, {\it On complex quadratic fields with class-number two}, Math. Comp. {\bf 29} (1975), 289--302.
	
	\bibitem{Vil97} F. R. Villegas, {\it Modular Mahler measures I}, Topics in number theory (University Park, PA, 1997), 17--48, Math. Appl. 467, Kluwer Acad. Publ. Dordrecht, 1999.
	
	\bibitem{Was08} L. C. Washington, {\it Elliptic curves. Number theory and cryptography}, 2nd ed., Chapman \& Hall/CRC, Boca Raton, FL, 2008. xviii+513 pp.
	
	\bibitem{Zud14} W. Zudlin, {\it Regulator of modular units and Mahler measures}, Math. Proc. Cambridge Philos. Soc. {\bf 156} (2014), no. 2, 313--326.
	
	\bibitem{database} The LMFDB Collaboration, {\it The $L$-functions and modular forms database}, \url{http://www.lmfdb.org}, 2022.

\end{thebibliography}
\end{document}